\documentclass[a4paper,english]{article}
\title{Random-field Solutions to Linear Hyperbolic Stochastic Partial Differential Equations with Variable Coefficients\thanks{The second author is supported by the grant MTM 2012-31192 from the Direcci\'on General de Investigaci\'on, Ministerio de Economia y Competitividad, Spain.  The research leading to this article has been carried out while the second author was on a Estanc\'\i{}a breve program at the University of Ferrara funded by EEBB-I-2013-06919 the Ministerio de Economia y Competitividad, Spain. The financial support and the hospitality by the staff at the University of Ferrara is gratefully acknowledged.}}

\author{Alessia Ascanelli\footnote{\sl Alessia Ascanelli, Dipartimento di Matematica ed Informatica, Universit\`a di Ferrara, Via Machiavelli n.~30, 44121 Ferrara, Italy, e-mail: alessia.ascanelli@unife.it}\ \  and Andr{\'e} S{\"u}\ss\footnote{\sl Andr\'e S\"u\ss, Departament de Probabilitat, L\`ogica i Estad\'istica, Gran Via, 585, 08007 Barcelona, Spain, e-mail: andre.suess@ub.edu}}
\date{}

\usepackage[utf8]{inputenc}
\usepackage{mathrsfs}
\usepackage{babel}
\usepackage{amsmath,amsthm,amssymb,amsxtra}
\usepackage{hyperref}
\usepackage{enumitem}
\usepackage{pdfsync}
\usepackage{colortbl}
\usepackage{url}

\newcommand*{\e}{\mathrm{e}}
\newcommand*{\ii}{\mathrm{i}}
\newcommand*{\dbar}{\mbox{\dj}}									

\newcommand*{\scrB}{\ensuremath{\mathscr{B}}}	
\newcommand*{\scrF}{\ensuremath{\mathscr{F}}}	

\newcommand*{\caC}{\ensuremath{\mathcal{C}}}		
\newcommand*{\caD}{\ensuremath{\mathcal{D}}}	  
\newcommand*{\caE}{\ensuremath{\mathcal{E}}}		
\newcommand*{\caF}{\ensuremath{\mathcal{F}}}		
\newcommand*{\caH}{\ensuremath{\mathcal{H}}}
\newcommand*{\caM}{\ensuremath{\mathcal{M}}}
\newcommand*{\caP}{\ensuremath{\mathcal{P}}}		
\newcommand*{\caR}{\ensuremath{\mathcal{R}}}		
\newcommand*{\caS}{\ensuremath{\mathcal{S}}}		

\newcommand*{\bfP}{\ensuremath{\mathbf{P}}}		

\newcommand*{\N}{\mathbb{N}}										
\newcommand*{\Z}{\mathbb{Z}}										
\newcommand*{\R}{\mathbb{R}}										
\newcommand*{\Rd}{{\mathbb{R}^d}}							

\newcommand*{\E}{\mathbb{E}}										
\renewcommand*{\P}{\mathbb{P}}									
\newcommand{\Infkt}[1]{1_{ #1 }}								
\newcommand{\id}{\mathrm{id}}										
\DeclareMathOperator{\supp}{supp}							

\numberwithin{equation}{section}
\allowdisplaybreaks

\theoremstyle{plain}
\newtheorem{lemma}{Lemma}[section]
\newtheorem{theorem}[lemma]{Theorem}
\newtheorem{proposition}[lemma]{Proposition}
\newtheorem{corollary}[lemma]{Corollary}
\newtheorem{assumption}[lemma]{Assumption}

\theoremstyle{definition}
\newtheorem{definition}[lemma]{Definition}
\newtheorem{remark}[lemma]{Remark}
\newtheorem{example}[lemma]{Example}

\usepackage{xcolor}
\definecolor{red}{rgb}{1,0,0}

\begin{document}
\maketitle

\begin{abstract}
In this article we show the existence of a random-field solution to linear stochastic partial differential equations whose partial differential operator is hyperbolic and has variable coefficients that may depend on the temporal and spatial argument. The main tools for this, pseudo-differential and Fourier integral operators, come from microlocal analysis. The equations that we treat are second-order and higher-order strictly hyperbolic, and second-order weakly hyperbolic with uniformly bounded coefficients in space. For the latter one we show that a stronger assumption on the correlation measure of the random noise might be needed. Moreover, we show that the well-known case of the stochastic wave equation can be embedded into the theory presented in this article.
\end{abstract}

{\bf 2010 Mathematics Subject Classification:} Primary: 35L10, 60H15; Secondary: 35L40, 35S30

\bigskip

{\bf Keywords:} stochastic partial differential equations, stochastic wave equation, hyperbolic partial differential equations, fundamental solution, variable coefficients, Fourier integral operators

\section{Introduction}
In the recent years there has been a huge progress in the solution theory to stochastic partial differential equations (SPDEs). A linear SPDE is given by the following equation
\begin{equation}\label{eq:SPDE}
  Lu(t,x) = \gamma(t,x) + \sigma(t,x)\dot{F}(t,x),
\end{equation}
where $L$ is a partial differential operator, $\gamma,\sigma:\R^{1+d}\to\R$ are functions, subject to certain regularity conditions and $F$ is a random noise term that will be described in detail in Section \ref{sec:stochastics}. Due to the singularity of the random noise, the sample paths of $u$ are in most situations not in the domain of the operator $L$. One way to make sense of this equation in the case of constant coefficients is the following: we define the solution to \eqref{eq:SPDE} as a sum of a deterministic term $I_0$ accounting for the initial conditions, a stochastic and a deterministic convolution of the terms on the right-hand side with the fundamental solution $\Lambda$ to the partial differential equation (briefly, PDE in the following) $Lu=0$: 
\begin{align}\label{eq:mildsolutionSPDE}
  u(t,x) = I_0(t,x)	& + \int_0^t\int_\Rd \Lambda(t-s,x-y)\sigma(s,y)M(ds,dy) \notag	\\
										& + \int_0^t\int_\Rd \Lambda(t-s,x-y)\gamma(s,y)dyds, 
\end{align}
where $M$ is the martingale measure derived from the random noise $\dot F$, see Section \ref{sec:stochastics}. Solutions of this type are called \emph{mild solutions} and were introduced in \cite{walsh} and later generalized in \cite{dalang,conusdalang}. Note that the solution $u$ is defined as a random variable for each $(t,x)\in[0,T]\times\Rd$, where $T>0$ is the time horizon of the equation. Due to this feature, we call these solutions \emph{random-field solutions} in contrast to function-valued solutions, which cannot be evaluated in the spatial argument, but only as a Hilbert- or Banach-space valued random element in the temporal argument, see \cite{dapratozabczyk} for that theory.

Many interesting properties of random-field solutions for SPDEs have been studied for the case when the partial differential operator $L$ has constant coefficients, e.g.\ the regularity of the probability measure induced by the solution \cite{pardouxzhang,nualartquer,sanzsuess1}, large deviation principles \cite{ortizsanz,marquezsarra3}, Varadhan estimates \cite{milletsanz,sanzsuess2}, support theorems \cite{milletsanz2000,delgadosanz}, path properties such as H{\"o}lder continuity \cite{sarrasanz,dalangsanz} and much more. See also the references in these works for a more detailed account. Due to the restriction on constant coefficients, the set of concrete examples for random-field solutions to SPDEs is essentially limited to the stochastic heat equation and the stochastic wave equation (possibly with lower order terms). Note furthermore, that in \cite{sanzvuillermot} the existence of a random-field solution to a class of parabolic equations with variable coefficients has been established. 

With the present article, we aim to study the case of hyperbolic equations with variable coefficients which has, to our knowledge, not been considered yet, producing a random-field solution similar to \eqref{eq:mildsolutionSPDE}, and so enlarging the class of SPDEs which admit a random-field solution.

Let us briefly explain the contents of this article. We consider linear SPDEs whose partial differential operators have variable coefficients that may depend on space and time. For these equations we want to derive conditions on the coefficients such that the notion of random-field solution makes sense. The equations that we will consider are general second- and higher-order hyperbolic equations on the whole space $\Rd$. The main result of this paper, Theorem \ref{thm:2ndorderhyperbolic}, shows that if the coefficients of the partial differential operator
\begin{equation}\label{intro1i}
	L=\partial_t^2 - \sum_{j,k=1}^d a_{j,k}(t,x)\partial_{x_j}\partial_{x_k} - \sum_{j=1}^d b_j(t,x)\partial_ {x_j} - c(t,x)
\end{equation}
are smooth and bounded with bounded derivatives of all orders in $x$ and continuous in $t$ (the $a_{j,k}$ have to be differentiable in $t$), and the operator $L$ is of strictly hyperbolic type, i.e.\ it satisfies:
\begin{equation}\label{intro2}
	\sum_{j,k=1}^d a_{j,k}(t,x)\xi_j\xi_k \geq C|\xi|^2,
\end{equation}
for all $t\in[0,T]$ and all $x,\xi\in\Rd$, then the concept of random-field solution for the SPDE \eqref{eq:SPDE} makes sense. The main tools for achieving this objective, e.g.\ pseudo-differential operators and Fourier integral operators, come from microlocal analysis. To our knowledge, this is the first time that their full potential is rigorously applied within the theory of random-field solutions to SPDEs. Note however the case of SPDEs with a pseudo-differential operator in the framework of function-valued solutions, see \cite{tindel}.

The paper is organized as follows.

In Section \ref{sec:stochastics} we review the notions of stochastic integration with respect to martingale measures and random-field solutions to SPDEs. Since, in contrast to the classic references \cite{walsh,dalang}, we do not assume the partial differential operator to have constant coefficients, its fundamental solution is no longer stationary in time and space, i.e.\ it cannot be written as $\Lambda(t-s,x-y)$ as in \eqref{eq:mildsolutionSPDE}, but rather as $\Lambda(t,s,x,y)$. This small difference will have some consequences on the conditions for the existence of random-field solutions to SPDEs with such partial differential operators. The new existence conditions are summarized in Theorem \ref{thm:existenceanduniquenessconstant} which is the main result of the Section.

Section \ref{sec:hSPDE} is devoted to applying the concepts developed in Section \ref{sec:stochastics} to hyperbolic SPDEs. In Section \ref{sec:microlocalanalysis} we present a quick introduction to microlocal analysis, where we gather all the tools necessary for constructing the fundamental solution to an hyperbolic equation. At this point we have to note that the concept of fundamental solutions to PDEs, which is used in the framework of random-field solutions to SPDEs (in Section \ref{sec:stochastics}), is different from the one that is used in microlocal analysis (in Section \ref{sec:microlocalanalysis}). The relation between the two concepts is outlined in Remark \ref{keyrem}:
the fundamental solution in the sense of Section \ref{sec:stochastics} is the Schwartz kernel of the fundamental solution in the sense of Section \ref{sec:microlocalanalysis}. The main result of the section is Proposition \ref{prop:ftofschwartzkernel}, where we calculate the Fourier transform with respect to the second variable of the Schwartz kernel of the fundamental solution.

In the subsequent Section \ref{sec:appendixhyperbolic} we present an overview over the procedure that we will use in the remaining sections to prove the mild solutions of hyperbolic SPDEs are well-defined.
For the proof, we follow the same ideas as in \cite{AC, AC1,AC2}, where fundamental solutions to deterministic hyperbolic PDEs have been obtained (see also the generalization of the construction of the fundamental solution given in the very recent paper \cite{ACor}, that will be needed in the forthcoming paper \cite{ACS} to extend the results of the present paper to hyperbolic equations with coefficients admitting a polynomial growth as $|x|\to\infty$). 

In Section \ref{sec:2ndorderhyperbolic} we show our main result that the operator in \eqref{intro1i} under condition \eqref{intro2} has a fundamental solution that satisfies all the necessary assumptions for the well-definedness of a random-field solution for \eqref{eq:SPDE}. 
Here the procedure is the following: we reduce the second-order equation to a first-order system, for which one can compute the fundamental solution explicitly, see \cite{kumano-go}. From the fundamental solution to the system we compute the fundamental solution to the second-order equation, and show that the conditions on the fundamental solution for the well-definedness of the stochastic and deterministic convolutions, and therefore on the existence of a random-field solution to the SPDE, are fulfilled. We conclude Section \ref{sec:2ndorderhyperbolic} with an example, where we show how the classic stochastic wave equation fits in the theory here presented.

In the subsequent two sections \ref{sec:higherorderhyperbolic} and \ref{sec:weakhyperbolic} we deal with generalizations of the second-order hyperbolic equations treated in Section \ref{sec:2ndorderhyperbolic}.

Section \ref{sec:higherorderhyperbolic} is devoted to strictly hyperbolic equations of higher order $n\in\N$, $n\geq2$; we show that the coefficients have to satisfy similar conditions as in the case of second-order strictly hyperbolic equations in order to obtain the well-definedness of a random-field solution to the SPDE. However, we will show that the higher the order of the equation, the larger is the class of spectral measures that we can allow for.

Finally, in Section \ref{sec:weakhyperbolic} we relax the assumption of the strict hyperbolicity on the partial differential operator and provide an example of an operator of the form \eqref{intro1i} which does not satisfy \eqref{intro2}. We show that, in this case, random-field solutions can only be obtained under stronger conditions, needed to
deal with this degeneracy.

Throughout this article, let for all $\xi\in\Rd$, $|\xi|:= (\sum_{j=1}^d \xi_j^2)^{1/2}$ and $\langle\xi\rangle:=(1+|\xi|^2)^{1/2}$. Let moreover $\alpha$ denote a multiindex with the usual arithmetic o\-pe\-ra\-tions. We will denote partial derivatives with $\partial$. Moreover, we set $D=-\ii\partial$, $\ii$ the imaginary unit, for the sake of Fourier transform. We will denote by $\caC^m(X)$, $\caC_b^m(X)$, $\caC^m_0(X)$, $\caS(X)$, $\caD(X)$, $\caM_b(X)$, $\caS'(X)$, $\caS'_r(X)$ and $\caD'(X)$ the $m$-times continuously differentiable functions, the $m$-times continuously differentiable functions with uniformly bounded derivatives of all orders $\leq m$, the $m$-times continuously differentiable functions with compact support, 
the Schwartz functions, the test functions, the complex-valued measures with finite total variation, the tempered distributions, the tempered distributions with rapid decrease and the distributions on some finite or infinite-dimensional space $X$, respectively. 

A tempered distribution $u$ is in $\caS'_r(\R^d)$ if $\forall k\in\Z$ $\langle\cdot\rangle^ku$ is a bounded distribution; the space of bounded distributions $\mathcal D'_{L^\infty}(\R^d)$ is the dual space of $\mathcal D_{L^1}(\R^d)$ (which is the space of all $C^\infty(\R^d)$ functions $f$ such that $\partial^\alpha f\in L^1(\R^d)$, $\forall\alpha\in \Z^n_+.$). 


We shall denote by $H^r(\R^d)$ the Sobolev space of order $r\geq 0$ on $L^2(\R^d)$. We will use the notation $\N_0:=\N\cup\{0\}$ and $\R^d_*:=\Rd\backslash\{0\}$. Let furthermore $C>0$ be a generic constant, whose value can change from line to line without further notice.

\section*{Acknowledgements} We wish to thank an anonymous referee, for the constructive criticism and the suggestions, aimed at improving the overall quality of the paper.

\section{Stochastic integration with respect to martingale measures and spatially non-homogeneous SPDE}\label{sec:stochastics}
In this section we introduce the framework to treat mild solutions to linear SPDEs similarly as in \eqref{eq:mildsolutionSPDE}. We explain how stochastic integration with respect to martingale measures is defined, collect some conditions on the integrands and provide a theorem for the well-definedness of mild solutions to SPDEs in the case of variable coefficients. The main novelty of this section compared with \cite{dalang,conusdalang} is that we do not make the assumption of spatial homogeneity. The price we pay is that we cannot treat semilinear SPDEs, see the comment at the end of this section. So let us consider the following mild formulation of the SPDE in \eqref{eq:SPDE}
\begin{align}\label{eq:mildsolutionSPDE2}
  u(t,x) = I_0(t,x)	& + \int_0^t\int_\Rd \Lambda(t,s,x,y)\sigma(s,y)M(ds,dy) \\
										& + \int_0^t\int_\Rd \Lambda(t,s,x,y)\gamma(s,y)dyds. \notag
\end{align}
This is the way in which we understand the SPDE \eqref{eq:SPDE}, and in the following we provide conditions to show that each term on the right-hand side of this equality is meaningful. In fact, we call "mild random field solution to \eqref{eq:SPDE}" a family of random variables $u(t,x)$, $(t,x)\in[0,T]\times\R^d$
defined by \eqref{eq:mildsolutionSPDE2}. 

Let in the following $\{F(\phi);\; \phi\in\mathcal{C}_0^\infty(\mathbb{R}_+\times\Rd)\}$ be a Gaussian process with mean zero and covariance functional given by
\begin{equation}
	\E[F(\phi)F(\psi)] = \int_0^\infty\int_\Rd \big(\phi(t)\star\tilde{\psi}(t)\big)(x)\Gamma(dx) dt,
	\label{eq:correlation}
\end{equation}
where $\tilde{\psi}(t,x) := \psi(t,-x)$, $\star$ is the convolution operator in the $x-$variable and $\Gamma$ is a nonnegative, nonnegative definite, tempered measure on $\Rd$. Then \cite[Chapter VII, Th\'{e}or\`{e}me XVIII]{schwartz} implies that there exists a nonnegative tempered measure $\mu$ on $\Rd$ such that $\caF\mu = \Gamma$, where $\caF$ denotes the Fourier transform given for functions $f\in L^1(\Rd)$ by
\begin{equation}\label{eq:definitionfouriertransform}
	(\caF f)(\xi) := \int_\Rd \e^{-\ii x\cdot\xi}f(x)dx,
\end{equation}
where $x\cdot\xi$ denotes the inner product in $\Rd$. We can then extend the Fourier transform to tempered distributions $T\in\caS'(\Rd)$ by the relation
\begin{equation}\label{eq:definitionfouriertransformSD}
	\langle \caF T,\phi\rangle = \langle T, \caF\phi\rangle,
\end{equation}
for all $\phi\in\caS(\Rd)$.

By Parseval's identity, the right-hand side of \eqref{eq:correlation} is equal to
\begin{equation*}
	\E[F(\phi)F(\psi)] = \int_0^{\infty}\int_{\Rd}\caF\phi(t)(\xi)\overline{\caF\psi(t)}(\xi)\mu(d\xi) dt.
\end{equation*}
As explained in \cite{dalangfrangos}, by approximating indicator functions with $\caC^\infty_0$-functions, the process $F$ can be extended to a worthy martingale measure $M=(M_t(A);\; t\in\R_+, A\in\scrB_b(\Rd))$ where $\scrB_b(\Rd)$ denotes the bounded Borel subsets of $\Rd$. The natural filtration generated by this martingale measure will be denoted in the sequel by $(\scrF_t)_{t\geq 0}$.

In the following we shall use \cite{walsh,dalang,conusdalang} as reference for an integration theory with respect to the martingale measure constructed above. Fix $T>0$. For stochastic processes $f$ and $g$, indexed by $(t,x)\in[0,T]\times\Rd$ and satisfying suitable conditions, we define the pre-inner product
\begin{align}
 \langle f,g\rangle_{0}
 & = \E\bigg[\int_0^T\int_\Rd \big(f(s)\star\tilde{g}(s)\big)(x)\Gamma(dx) ds \bigg] \label{eq:norm01}\\
 & = \E\bigg[\int_0^T\int_\Rd \caF f(s)(\xi)\overline{\caF g(s)(\xi)}\mu(d\xi) ds\bigg], \label{eq:norm02}
 \end{align}
where the corresponding semi-norm $\|\cdot\|_0$ is defined in the usual way. Moreover, we define the semi-norm
\[ \|f\|_+^2 := \E\bigg[\int_0^T\int_\Rd \big(|f(s)|\star|\tilde{g}(s)|\big)(x)\Gamma(dx) ds \bigg]. \]

Let $\caE$ denote the set of {\it simple processes} $g$, that is, stochastic processes of the form
\[ g(t,x;\omega) = \sum_{j=1}^m 1_{(a_j,b_j]}(t)1_{A_j}(x)X_j(\omega), \]
for some $m\in\N$, where $0\leq a_j < b_j\leq T$, $A_j\in\scrB_b(\Rd)$ and $X_j$ is a bounded and $\scrF_{A_j}$-measurable random variable for all $1\leq j\leq n$. The stochastic integral of $g$ with respect to the martingale measure $M$, denoted by $g\cdot M$ is given by
\[ (g\cdot M)_t := \sum_{j=1}^m \big(M_{t\wedge b_j}(A_j) - M_{t\wedge a_j}(A_j)\big)X_j, \]
where $x\wedge y := \min\{x,y\}$. One can show by appying the definition that
\begin{equation}\label{eq:isometry}
  \E\big[(g\cdot M)_t^2\big] = \|g\|^2_0,
\end{equation}
for all $g\in\caE$. Following \cite{dalang}, we denote by $\caP_0$ the completion of $\caE$ with respect to $\langle\cdot,\cdot\rangle_0$. Then $\caP_0$ is a Hilbert space consisting of predictable processes which may contain tempered distributions in the $x$-argument (whose Fourier transform are functions, $\P$-almost surely).  The norm in this space is given by the $\|\cdot\|_0$-norm defined above in \eqref{eq:norm02} and for sufficiently smooth elements of $\caP_0$, this norm can be also written as in \eqref{eq:norm01}. Note that $\caP_0$ is not defined as the set of predictable processes $g$ for which $\|g\|_0 < \infty$.
In fact, it can be shown that the latter space is not complete. So we have that $\caP_0$ is the space of all integrable (with respect to $M$) processes and the stochastic integrals are defined as an $L^2(\Omega)$-limit of simple processes via the isometry \eqref{eq:isometry}. In \cite[Lemma 2.2]{sanzsuess1}, it was shown that $\caP_0=L^2_{p}([0,T]\times\Omega,\caH)$,
where here $L^2_p(\ldots)$ stands for the predictable stochastic processes in $L^2(\ldots)$ and $\caH$ is the Hilbert space which is obtained from completing the Schwartz functions with respect to the inner product $\langle\cdot,\cdot\rangle_0$.

On the other hand, we define $\caP_+$ to be the set of all predictable processes for which $\|g\|_+<\infty$. Then
$\caP_+$ is a Banach space and the simple processes are dense in this space, see \cite[Proposition 2.3]{walsh}. Note that since $\|\cdot\|_0\leq\|\cdot\|_+$, we have $\caP_+\subset\caP_0$, and this inclusion may be strict.

Now we describe how to integrate time- and space-dependent integrands of a special form $\Lambda\sigma$ into the SPDE \eqref{eq:SPDE}. Here, $\Lambda$ is the fundamental solution to the associated PDE and $\sigma$ is the coefficient on the right-hand side of the SPDE depending on the time and space parameter. That means we want to make sense of the stochastic integral
\begin{equation}\label{eq:integralLambdasigma}
	\int_0^t\int_\Rd \Lambda(t,s,x,y)\sigma(s,y)M(ds,dy).
\end{equation}
Note that in this integral and throughout this article we write $\Lambda(t,s,x,y)$ although this object will be a distribution in the last argument. This abuse of notation is for the sake of briefness.

With the help of \eqref{eq:isometry} we calculate the second moment of \eqref{eq:integralLambdasigma}, from where we can deduce sufficient conditions for its well-definedness. We will assume in the following, compare with Assumption \ref{one!}, that the spatial Fourier transform of the function $\sigma$ is a complex-valued measure with finite total variation, i.e.\ for all $s\in[0,T]$
\[ |\caF\sigma(\cdot,s)| = |\caF\sigma(\cdot,s)|(\Rd) = \sup_\pi \sum_{A\in\pi} |\caF\sigma(\cdot,s)|(A) < \infty, \]
where $\pi$ is any partition on $\Rd$ into measurable sets $A$, and the supremum is taken over all such partitions. Let throughout the remainder of the article $\nu_s:=\caF\sigma(\cdot,s)$, and let $|\nu_s|$ denote its total variation.

Now we compute the norm on the right-hand of \eqref{eq:isometry}. Using \eqref{eq:isometry}, the definition of convolution between a distribution $\nu_s$ and a function $\caF(\Lambda(t,s,x,\cdot))$, the well-known fact that the Fourier transform of a product is the convolution of the Fourier transforms, and Minkowski's integral inequality we obtain
\begin{align}\label{eq:normbound}
  & \|\Lambda(t,\cdot,x,\ast)\sigma(\cdot,\ast)\|^2_0 \notag\\
  & = \int_0^t\int_\Rd\int_\Rd \Lambda(t,s,x,y)\Lambda(t,s,x,y-z)\sigma(s,y)\sigma(s,y-z)dy\Gamma(dz)ds \notag\\
  & = \int_0^t\int_\Rd \big|\caF(\Lambda(t,s,x,\cdot)\sigma(s,\cdot))(\xi)\big|^2\mu(d\xi)ds \notag \\
  & = \int_0^t\int_\Rd\bigg|\int_\Rd \caF\Lambda(t,s,x)(\xi-\eta)\nu_s(d\eta)\bigg|^2\mu(d\xi)ds \notag\\
  &\leq \int_0^t\bigg(\int_\Rd\bigg(\int_\Rd |\caF\Lambda(t,s,x)(\xi-\eta)|^2\mu(d\xi)\bigg)^{1/2}|\nu_s(d\eta)|\bigg)^2ds\notag\\
  & \leq \int_0^t\bigg(\sup_{\eta\in\Rd} \int_\Rd |\caF\Lambda(t,s,x)(\xi+\eta)|^2 \mu(d\xi)\bigg) |\nu_s|^2ds.
\end{align}
We see that $\caF\sigma(s)$ has to have finite total variation for almost all $s\in[0,T]$ in order that the previous term be finite. We will now also assume that $\sigma(s)\in C_b(\Rd)$ for every $s\in[0,T]$ in order to simplify an argument later. In fact, this condition follows directly from the fact that $\nu_s$ has finite total variation for all $s\in[0,T]$ by considering the inverse Fourier transform (in the distributional sense) of $\nu_s$ and reminding that the Fourier transform of a measure with finite total variation is uniformly continuous, see \cite[Chapter 5, \S 26]{bill}.

Let in the following $\Delta_T$ be the simplex given by $0\leq t\leq T$ and let $0<s<t$. In order for the well-definedness of the stochastic integral, we need to assume the following.

\begin{assumption}\label{one!}
For  $(t,s,x)\in\Delta_T\times\Rd$, let $\Lambda(t,s,x)$ be a deterministic function with values in $\caS'_r(\Rd)$ and let $\sigma$ be a function in $L^2([0,T],C_b(\Rd))$ such that:
 \begin{enumerate}[label=\bfseries (A\arabic{enumi}),ref=\bfseries (A\arabic{enumi})]
    \item\label{ass:A1} the function $(t,s,x,\xi)\mapsto\caF\Lambda(t,s,x)(\xi)$ is measurable, the function $s\mapsto\caF\sigma(s)=\nu_s \in L^2([0,T],\caM_b(\Rd))$, and moreover
 		\begin{equation}\label{eq:condition1}
			\int_0^T\bigg(\sup_{\eta\in\Rd} \int_\Rd |\caF\Lambda(t,s,x)(\xi+\eta)|^2 \mu(d\xi)\bigg) |\nu_s|^2 ds < \infty.
		\end{equation}
    \item\label{ass:A2} $\Lambda$ and $\sigma$ are as in \ref{ass:A1} and
			\begin{align*}
				\lim_{h\downarrow0} \int_0^T		& \bigg(\sup_{\eta\in\Rd} \int_\Rd \sup_{r\in(s,s+h)} |\caF(\Lambda(t,s,x)-\Lambda(t,r,x))(\xi+\eta)|^2 \mu(d\xi)\bigg) |\nu_s|^2ds = 0.
			\end{align*}
  \end{enumerate}
\end{assumption}

In the case where the coefficient $\sigma$ does not depend on the spatial argument, these assumptions take on the following form.

\begin{example}\label{ex:simplecase}
  If $\sigma$ does not depend on the spatial argument, then
  \[\caF\sigma(s) = (2\pi)^d\sigma(s)\delta_0, \]
  where $\delta_0$ is the Dirac delta distribution with total variation $1$. Doing the same computations as in \eqref{eq:normbound}, we arrive at the necessary condition that
  \[ \int_0^T\sigma(s)^2\int_\Rd |\caF\Lambda(t,s,x)(\xi)|^2 \mu(d\xi)ds < \infty, \]
  which is actually weaker than \eqref{eq:condition1} in the sense that there is no supremum over $\eta$.
\end{example}

The reason for the assumption that $\Lambda(t)\in\caS'_r(\Rd)$ is that in this case the Fourier transform in the second spatial argument is a smooth function of slow growth, and the convolution of a distribution in $\caS'_r(\Rd)$ with any other distribution in $\mathcal S'(\R^d)$ is well-defined, see \cite[Chapter VII, \S 5]{schwartz} and \cite{ACSprime} for more detalis. A necessary and sufficient condition for $T\in\caS'_r(\Rd)$ is that each regularization of $T$ with a $\caC^\infty_0$-function is a Schwartz function. This will be true in our case due to Proposition \ref{prop:ftofschwartzkernel} and the fact that the Fourier transform is a bijection on the Schwartz functions, see Lemma \ref{lem:FTfundamentalsolution}.

Notice that Assumption \ref{one!} generalizes to the non spatially homogeneous case the corresponding conditions (H2) and (3.6) in \cite{conusdalang}. 
In the spatially homogeneous case investigated in \cite{dalang,conusdalang}, when $\Lambda$ is the solution of the heat or the wave equation, \cite[Lemma 6.1]{sanzbook} shows that, the corresponding condition \eqref{eq:condition1} is equivalent to
\begin{equation}\label{eq:condition3}
  \int_\Rd \frac{1}{1+|\xi|^2}\mu(d\xi) < \infty.
\end{equation}
So one of our aims in the subsequent sections is to find a similar estimate for \eqref{eq:condition1}, which reads
\begin{equation}\label{eq:condition2}
  \sup_{\eta\in\Rd}\int_\Rd \frac{1}{(1+|\xi+\eta|^2)^\kappa}\mu(d\xi) < \infty,
\end{equation}
for some $\kappa\in(0,1]$. Note that if the correlation measure $\Gamma$ is absolutely continous, then condition \eqref{eq:condition2} is equivalent to \eqref{eq:condition3}, see \cite{peszat}.

 In contrast to the methods used in the proof of \cite[Lemma 6.1]{sanzbook}, \eqref{eq:condition2} will follow easily from a quick investigation of the order of the symbol associated to the fundamental solution using the tools presented in Section \ref{sec:microlocalanalysis}.

We can now prove, similarly to \cite[Theorem 3.1]{conusdalang} that under the two assumptions above, the stochastic integral is well-defined.
\begin{theorem}\label{thm:existencestochasticintegral}
  Under Assumption \ref{one!}, we have that $\Lambda\sigma\in\caP_0$. In particular, the stochastic integral $\int_0^t\int_\Rd \Lambda(t,s,x,y)\sigma(s,y)M(ds,dy)$ is well-defined and
 \begin{align*}
	& \E\big[((\Lambda(t,\cdot,x,\ast)\sigma(\cdot,\ast))\cdot M)_t^2\big] \\
	&	\quad \leq \int_0^t\bigg(\sup_{\eta\in\Rd} \int_\Rd |\caF\Lambda(t,s,x)(\xi+\eta)|^2 \mu(d\xi)\bigg) |\nu_s|^2ds.
\end{align*}
\end{theorem}
\begin{proof}
Fix throughout this proof $(t,x)\in[0,T]\times\Rd$, and let $s\in[0,t)$. Take $\psi\in\caC^\infty_0(\Rd)$ such that $\supp\psi\subseteq B_d(0,1)$ (the unit ball in $\Rd$). Then set for all $n\in\N$, $\psi_n(y):=n^d\psi(ny)$ and $\Lambda_n(t,s,x):=\Lambda(t,s,x)\ast\psi_n$. Then we have $|\caF\psi_n(\xi)|\leq1$, $|\caF\psi_n(\xi)|\to1$ pointwise, and $|\caF\Lambda_n(t,s,x)(\xi)|=|\caF\Lambda(t,s,x)(\xi)||\caF\psi_n(\xi)|$. If we have that $\Lambda_n(t,\cdot,x,\ast)\sigma(\cdot,\ast)\in\caP_0$ for all $n\in\N$, then performing the same steps as in \eqref{eq:normbound} yields
\begin{align*}
  & \|(\Lambda(t,\cdot,x,\ast) - \Lambda_n(t,\cdot,x,\ast))\sigma(\cdot)\|_0^2 \\
  & \quad\leq \int_0^t\int_\Rd\bigg(\int_\Rd |\caF(\Lambda(t,s,x)-\Lambda_n(t,s,x))(\xi+\eta)||\nu_s(d\eta)|\bigg)^2\mu(d\xi) ds \\
  & \quad\leq \int_0^t\int_\Rd\bigg(\int_\Rd | \caF\Lambda(t,s,x)(\xi+\eta)|\cdot |1-\caF\psi_n(\xi+\eta)||\nu_s(d\eta)|\bigg)^2\mu(d\xi) ds;
\end{align*}
the latter term goes to zero as $n\to\infty$, since we have $|\caF\psi_n(\xi)|\to1$, thanks to the bounded convergence theorem: indeed $|1-\caF\psi_n(\xi+\eta)|\leq2$ so repeating the same computations as in \eqref{eq:normbound} we get
\begin{align*}
&\int_\Rd\bigg(\int_\Rd | \caF\Lambda(t,s,x)(\xi+\eta)|\cdot |1-\caF\psi_n(\xi+\eta)||\nu_s(d\eta)|\bigg)^2\mu(d\xi)\\
&\quad\leq  \int_\Rd\bigg(\int_\Rd 2|\caF\Lambda(t,s,x)(\xi+\eta)||\nu_s(d\eta)|\bigg)^2\mu(d\xi)\\
&\quad\leq  4\bigg(\int_\Rd\bigg(\int_\Rd |\caF\Lambda(t,s,x)(\xi+\eta)|^2\mu(d\xi)\bigg)^{1/2}|\nu_s(d\eta)|\bigg)^2\\
&\quad\leq 4\bigg( \sup_{\theta\in\Rd}\int_\Rd |\caF\Lambda(t,s,x)(\xi+\theta)|^2\mu(d\xi)\bigg)|\nu_s|^2\\
\end{align*}
which does not depend on $n$ and is in $L^1[0,T]$ by assumption \ref{ass:A1}.
\\
In order to show that $\Lambda_n(t,\cdot,x,\ast)\sigma(\cdot,\ast)\in\caP_0$, we define
\[ \Lambda_{n,m}(t,s,x,y) := \sum_{j=0}^{2^m-1} \Lambda_n(t,t_m^{j+1},x,y)\Infkt{[t_m^{j},t_m^{j+1})}(s), \]
for all $m\in\N$, where $t_m^j = jT2^{-m}$. Then $\Lambda_{n,m}(t,s,x,\ast)\in\caS(\Rd)$ and
\begin{align*}
  & \|\Lambda_{n,m}(t,\cdot,x,\ast)\sigma(\cdot,\ast)\|^2_+ \\
  & = \int_0^t\int_\Rd\int_\Rd |\Lambda_{n,m}(t,s,x,y)||\Lambda_{n,m}(t,s,x,y-z)||\sigma(s,y)||\sigma(s,y-z)|dy\Gamma(dz)ds \\
  & \leq \int_0^t \|\sigma(s)\|^2_{L^\infty(\Rd)} \bigg(\int_\Rd |\Lambda_{n,m}(t,s,x,y-z)||\Lambda_{n,m}(t,s,x,y)|dy\Gamma(dz)\bigg)ds.
\end{align*}
This is the (only) place, where we have used that $\sigma(s)\in L^\infty(\Rd)$ for almost all $s\in[0,T]$. Now Leibniz' formula \cite[Exercise 26.4]{treves} implies that for each $s\in[0,t)$, the term in the parentheses in the last line of the previous inequality is finite. Moreover, since $ \Lambda_{n,m}(t,\cdot,x,\ast)$ was a step function in $s$, it is also uniformly bounded and the fact that $\sigma(s)\in C_b(\Rd)$ for every $s\in[0,T]$ together with the assumption on $\sigma$ implies the finiteness of this term. Therefore $\Lambda_{n,m}\sigma\in\caP_+$, which implies that there exists a sequence of step functions approximating this object.

The last step in this proof is to show that $\Lambda_{n,m}(t,\cdot,x,\ast)\sigma(\cdot,\ast)$ converges to $\Lambda_n(t,\cdot,x,\ast)\sigma(\cdot,\ast)$ in $\caP_0$ for all $(t,x)\in[0,T]\times\Rd$. We compute using \eqref{eq:normbound}
\begin{align*}
  & \|(\Lambda_{n,m}(t,\cdot,x,\ast)-\Lambda_n(t,\cdot,x,\ast))\sigma(\cdot,\ast)\|_{0}^2\leq  \\
  & \leq \int_0^t\bigg(\sup_{\eta\in\Rd} \int_\Rd |\caF(\Lambda_n(t,s,x)-\Lambda_{n,m}(t,s,x))(\xi+\eta)|^2 \mu(d\xi)\bigg)|\nu_s|^2ds \\
  & \leq \int_0^t\bigg(\sup_{\eta\in\Rd} \int_\Rd \sup_{r\in(s,s+T2^{-m})} |\caF(\Lambda_n(t,s,x)-\Lambda_n(t,r,x))(\xi+\eta)|^2 \mu(d\xi)\bigg)|\nu_s|^2ds,
\end{align*}
which goes to zero by \ref{ass:A2}, which ends the proof.
\end{proof}

Now we treat the pathwise integral in \eqref{eq:mildsolutionSPDE2}. Similar to the stochastic integral we first compute an estimate for its second moment from which we can deduce suitable sufficient conditions for its existence. We assume that the spatial Fourier transform of the coefficient $\gamma(s)$ is a measure with finite total variation, denoted by $\chi_s$. We obtain the following
\begin{align}\label{eq:normcontinuitypathwise}
  \bigg(\int_0^t\int_\Rd \Lambda(t,s,x,y)\gamma(s,y)dyds\bigg)^2
  & \leq T\int_0^t\big(\caF\big(\Lambda(t,s,x,\ast)\gamma(s,\ast)\big)(0)\big)^2ds \notag\\
  & \leq C\int_0^t\bigg(\int_\Rd \caF\Lambda(t,s,x,-\eta)\chi_s(d\eta)\bigg)^2ds \notag\\
  & \leq C\int_0^t\bigg(\sup_{\eta\in\Rd} |\caF\Lambda(t,s,x)(\eta)|^2\bigg)|\chi_s|^2 ds.
\end{align}

In order to give a rigorous meaning to the pathwise integral, we assume the following.

\begin{assumption}\label{three!}
For $(t,s,x)\in\Delta_T\times\Rd$, let $\Lambda(t,s,x)$ be a deterministic function with values in $\caS'_r(\Rd)$ and let 
$\gamma\in L^2([0,T],C_b(\Rd))$
 such that
  \begin{enumerate}[label=\bfseries (A\arabic{enumi}),ref=\bfseries (A\arabic{enumi})]\setcounter{enumi}{2}
    \item\label{ass:A3} the function $(t,s,x,\xi)\mapsto\caF\Lambda(t,s,x)(\xi)$ is measurable, the function $s\mapsto\caF \gamma(s)=\chi_s \in L^2([0,T],\caM_b(\Rd))$, and moreover
		\begin{equation}\label{eq:condition4}
			\int_0^T\bigg(\sup_{\eta\in\Rd} |\caF\Lambda(t,s,x)(\eta)|^2 \bigg) |\chi_s|^2ds < \infty.
		\end{equation}
		  \item\label{ass:A4} Let $\Lambda$ and $\gamma$ be as in \ref{ass:A3}
		\begin{align*}
		\lim_{h\downarrow0} & \int_0^T\bigg(\sup_{\eta\in\Rd} \sup_{r\in(s,s+h)} |\caF(\Lambda(t,s,x)-\Lambda(t,r,x))(\eta)|^2 \bigg)|\chi_s|^2	ds = 0.
		\end{align*}
  \end{enumerate}
\end{assumption}

Similar to Example \ref{ex:simplecase}, we can weaken the assumptions \ref{ass:A3} and \ref{ass:A4} when the coefficient $\gamma$ does not depend on the spatial argument. 

Note that the two conditions \ref{ass:A3} and \ref{ass:A4} coincide with \ref{ass:A1} and \ref{ass:A2} respectively if $\mu=\delta_0$. 

Under Assumption \ref{three!}, the pathwise integral is well-defined and \eqref{eq:normcontinuitypathwise} holds; to show this, it is sufficient to repeat the very same arguments as in the proof of Theorem \ref{thm:existencestochasticintegral}, replacing $\mu$ by $\delta_0$.

We now make a last assumption on the first term $I_0$ in \eqref{eq:mildsolutionSPDE2}, that accounts for the inital conditions.

\begin{assumption}\label{five!}
  \begin{enumerate}[label=\bfseries (A\arabic{enumi}),ref=\bfseries (A\arabic{enumi})]\setcounter{enumi}{4}
		\item\label{ass:A5} For every $(t,x)\in[0,T]\times\Rd$, $I_0(t,x)$ is finite.
  \end{enumerate}
\end{assumption}

With all these preparations we can now state the existence theorem for stochastic partial differential equations which are nonhomogeneous in space.

\begin{theorem}\label{thm:existenceanduniquenessconstant}
  Under Assumptions \ref{one!}, \ref{three!} and \ref{five!}, the random-field solution of the  SPDE \eqref{eq:SPDE}, which is given in \eqref{eq:mildsolutionSPDE2}, makes sense.
\end{theorem}
\begin{proof}
  We calculate the second moment of $u(t,x)$ in \eqref{eq:mildsolutionSPDE2} for any fixed $(t,x)\in[0,T]\times\Rd$ and obtain
  \begin{align*}
    & \E\big[u(t,x)^2\big] \\
    & \leq C\bigg(I_0(t,x)^2 + \E\bigg[\bigg(\int_0^t\int_\Rd \Lambda(t,s,x,y)\sigma(s,y)M(ds,dy)\bigg)^2\bigg] \\
    & \phantom{mmm} + \bigg(\int_0^t\int_\Rd \Lambda(t,s,x,y)\gamma(s,y)dyds\bigg)^2\bigg)\\
    & \leq C\bigg(I_0(t,x)^2 + \int_0^t\sup_{\eta\in\Rd} \int_\Rd |\caF\Lambda(t,s,x)(\xi+\eta)|^2\mu(d\xi)|\nu_s|^2ds \\
    & \phantom{mmm} + \int_0^t\sup_{\eta\in\Rd} |\caF\Lambda(t,s,x)(\eta)|^2|\chi_s|^2ds\bigg),
  \end{align*}
which is finite by assumption, so that $u(t,x)$ is well-defined as a random variable in $L^2(\Omega)$ for every $(t,x)\in[0,T]\times\Rd$.
\end{proof}

Note that if in the previous inequality all the terms on the right-hand side can be uniformly bounded in $t$ and $x$, then we have for the solution that
\[ \sup_{(t,x)\in[0,T]\times\Rd} \E\big[u(t,x)^2\big] < \infty. \]
This condition is essential in order to treat semilinear SPDEs. In fact, in this situation one can reproduce the results from \cite{nualartquer} that if the fundamental solution is a function or a nonnegative distribution, one can incorporate coefficients $\sigma$ and $\gamma$ which depend on the solution $u$. However, if the fundamental solution is only a general distribution as in \cite{conusdalang}, one also needs a stationarity condition which is not satisfied in the case when the partial differential operator has variable coefficients. Therefore we keep our attention to the linear case, because we cannot tell from the methods presented in Section \ref{sec:microlocalanalysis} whether the fundamental solution is a function, a nonnegative distribution or a distribution.


\section{Application to linear hyperbolic SPDEs}\label{sec:hSPDE}
Here we present an application of the integration theory presented in the previous section. We treat a variety of linear hyperbolic SPDEs.

\subsection{Microlocal Analysis}\label{sec:microlocalanalysis}
In this section we collect some results which will be useful to us when we construct the fundamental solution to hyperbolic equations in the sections below. Our main tool will be the Fourier transform $\caF$, which was defined in \eqref{eq:definitionfouriertransform} for all functions $f\in L^1(\Rd)$ and extended to Schwartz distributions in \eqref{eq:definitionfouriertransformSD}. The inverse of the Fourier transform can be given by
\[ (\caF^{-1} f)(x) := (2\pi)^{-d}\int_\Rd \e^{\ii x\cdot\xi}f(\xi)d\xi = (2\pi)^{-d}(\caF f)(-x),  \]
for all $f\in L^1(\Rd)$, and extended to Schwartz distributions, so that for all $T\in\caS'(\Rd)$ we have $\caF^{-1}\caF T=T$.

Using the Fourier transform we can now define Fourier integral operators, that is the operators we will need in order to construct the fundamental solution. For their definition we need the following ingredients: symbols and phase functions.

\begin{definition}[Symbols]\label{CPsymbol}
Let $m\in\R$. A $\caC^\infty(\Rd\times\R_*^d)$-function $p$ is called a \emph{symbol} of class $S^m$ if for every $R>0$ and for all $\alpha,\beta\in\N^d_0$ there exists a constant $C_{R,\alpha,\beta}>0$ such that
\[ |\partial^{\alpha}_{\xi}\partial^{\beta}_{x}p(x,\xi)| \leq C_{R,\alpha,\beta}\langle\xi\rangle^{m-|\alpha|}, \]
for all $x,\xi\in\Rd$ with $|\xi|\geq R$. We say that $p$ is a symbol of order $m$.
 \end{definition}		
We can then define $S^{-\infty}:=\cap_{m\in\R}S^m$ and $S^{+\infty}:=\cup_{m\in\R}S^m$; we trivially have that for every $m_1\leq m_2$ it holds $S^{-\infty}\subset S^{m_1}\subset S^{m_2}\subset S^{+\infty}$.
The space $S^m$ endowed with the family of seminorms $(|\cdot|_{l,R}^{(m)};l\in\N_0, R>0)$ defined by
		\[ |p|_{l,R}^{(m)} := \max_{|\alpha+\beta|\leq l}\sup_{x\in\Rd,\; |\xi|\geq R} |\partial^{\alpha}_{\xi}\partial^{\beta}_{x}p(x,\xi)|\langle\xi\rangle^{-m+|\alpha|} \]
becomes a Fr\'echet space and for any $p\in S^m$, we have by definition
\begin{equation}\label{eq:pointwiseestimateold}
	|\partial^{\alpha}_{\xi}\partial^{\beta}_{x}p(x,\xi)| \leq |p|^{(m)}_{|\alpha+\beta|,R}\langle\xi\rangle^{m-|\alpha|},\qquad \forall x\in\R^d,\ |\xi|\geq R
\end{equation}
where $|p|^{(m)}_{|\alpha+\beta|,R}$ is the smallest constant assuring \eqref{eq:pointwiseestimateold}. In this paper we denote the special case $|p|^{(m)}_{l,1}$ by $|p|^{(m)}_{l}$.

\begin{definition}[Asymptotic expansion]
  Let $(p_j)_{j\in\N}$ be a sequence of symbols $p_j\in S^{m_j}$, where $(m_j)_{j\in\N}$ is a nonincreasing sequence with $m_j\to-\infty$ as $j\to\infty$. Then we say that a symbol $p\in S^m$ has the \emph{asymptotic expansion} $p \sim \sum_{j=1}^\infty p_j$,
  if for any integer $n\in\N$
  \begin{equation}\label{eq:assexp2}
		p - \sum_{j=1}^{n-1}p_j \in S^{m_n}.
  \end{equation}
\end{definition}

Note that this concept does not imply the convergence of the formal series $\sum_{j=1}^\infty p_j$ in any sense, although the order of the difference in \eqref{eq:assexp2} goes to $-\infty$.

It is possible to show, see \cite[Theorem 4.2 page 152]{chazpir}, that every symbol $p\in S^m$ is uniquely determined (modulo an element of $S^{-\infty}$) by its formal series.

\begin{remark} An equivalent (modulo $S^{-\infty}(\Rd\times\R_\ast^d)$) definition of the class of symbols is the following: a $\mathcal C^\infty(\R^{2d})$-function $p$ belongs to $S^m(\R^{2d})$, $m\in\R$, if for all $\alpha,\beta\in\N_0^d$ there exists $C_{\alpha,\beta}>0$ such that 
$$
|\partial^{\alpha}_{\xi}\partial^{\beta}_{x}p(x,\xi)| \leq C_{\alpha,\beta}\langle\xi\rangle^{m-|\alpha|}, \forall x,\xi\in\Rd.
$$
Indeed, clearly $S^m(\R^{2d})\subset S^m(\R^{d}\times\R_\ast^d)$; conversely if $p\in S^m(\R^{d}\times\R_\ast^d)$ we can take a cut-off function $\chi\in C^\infty(\R^d)$ such that $\chi(\xi)=0$ for $|\xi|\leq1/2$ and $\chi(\xi)=1$ for $|\xi|\geq1$, and write it as $p=\chi p+(1-\chi)p=q+r,$ where $q=\chi p\in S^m(\R^{2d})$ and $r=(1-\chi)p\in S^{-\infty}(\R^d\times\R_\ast^d)$.
\\
$S^m(\R^{2d})$ is a Fr\'echet space with seminorms 
		\[ |p|_{l}^{(m)} := \max_{|\alpha+\beta|\leq l}\sup_{x,\xi\in\Rd} |\partial^{\alpha}_{\xi}\partial^{\beta}_{x}p(x,\xi)|\langle\xi\rangle^{-m+|\alpha|} \]
and for any $p\in S^m(\R^{2d})$ we have 
\begin{equation}\label{eq:pointwiseestimate}|\partial^{\alpha}_{\xi}\partial^{\beta}_{x}p(x,\xi)| \leq |p|^{(m)}_{|\alpha+\beta|}\langle\xi\rangle^{m-|\alpha|},\qquad \forall x,\xi\in\R^d.
\end{equation}
\end{remark}

\begin{definition}[Phase functions]\label{def:phasefunction}
	A \emph{phase function} is a $\caC^\infty$-function $\varphi:\R^{d_1}\times\R^{d_2}\to\R$ that is homogeneous of degree one in the second argument, i.e.\ $\varphi(x,t\xi)=t\varphi(x,\xi)$ for all $t>0$, and $\nabla_{x,\xi}\varphi(x,\xi)\neq0$ in $\R^{d_1}\times\R_*^{d_2}$.
\end{definition}

\begin{definition}[Oscillatory integral distribution]
  Let $\varphi$ be a phase function and $p\in S^m$. Then the \emph{oscillatory integral distribution} of $\e^{\ii\varphi(x,\cdot)}p(x,\cdot)$ is the distribution defined for all test functions $v\in\caD(\R^{d_1})$ by
	\begin{equation}\label{eq:oscillatoryintegral}
		\bigg\langle O_S-\int_{\R^{d_2}} \e^{\ii\varphi(\cdot,\xi)}p(\cdot,\xi)\dbar\xi,v\bigg\rangle = \int_{\R^{d_2}\times\R^{d_1}} \e^{\ii\varphi(x,\xi)}p(x,\xi)v(x)dx\dbar\xi,
	\end{equation}
where $\dbar\xi := (2\pi)^{-d_2}d\xi$ and the integral on the right hand side is convergent at least in the sense of \emph{oscillatory integrals}, see \cite{hormander}.
\end{definition}

With all this we can now define the so-called \emph{Fourier integral operators} and the subclass of \emph{pseudo-differential operators}.

\begin{definition}[Fourier integral operators]\label{def:FIO}
Let $\phi$ be a $\caC^\infty$-function on $\Rd\times\Rd$, homogeneous of degree one with respect to $\xi$ and $p\in S^m$. A \emph {Fourier integral operator} $P_\phi:\caS(\Rd)\to\caS(\Rd)$ with phase function $\phi$ and symbol $p$ is defined by
	\begin{align} \label{eq:definitionFIO}
		(P_\phi v)(x)	&	= \int_\Rd \e^{\ii\phi(x,\xi)}p(x,\xi) \caF v(\xi)\dbar\xi 	= \int_\Rd\int_\Rd \e^{\ii\phi(x,\xi)-\ii y\cdot\xi}p(x,\xi)v(y)dy\dbar\xi,
	\end{align}
	for all $v\in\caS(\Rd)$. We will write $P_\phi=P_\phi(x,D_x)=p_\phi(x,D_x)$ to denote a Fourier integral operator with phase function $\phi$ and symbol $p$.
\end{definition}

Note that strictly speaking, $\phi$ in the above definition is not a phase function in the sense of Definition \ref{def:phasefunction}, but the function $\varphi(x,y,\xi):=\phi(x,\xi)-y\cdot\xi$ is indeed a phase function both with respect to the three arguments $(x,y,\xi)$ and with respect to the two arguments $(y,\xi)$. Therefore the oscillatory integral converges and we will refer in the following to $\phi$ as a ``phase function``.

Now we provide a few examples for Fourier integral operators, which will be used throughout this article.

\begin{example}\label{ex:symbols}
With the basic choice of $\phi(x,\xi)=x\cdot\xi$ as phase function, we have
\begin{equation}\label{pseudodef}
(P_\phi u)(x) = \int_\Rd \e^{\ii x\cdot\xi}p(x,\xi)\caF u(\xi)\dbar\xi = \int_\Rd\int_\Rd \e^{\ii (x-y)\cdot\xi}p(x,\xi)u(y)dy\dbar\xi.
\end{equation}
Operators of the form \eqref{pseudodef} are called \emph{pseudo-differential operators}; given a symbol $p\in S^m$ we denote by $P(x,D_x)=p(x,D_x)$ a pseudo-differential operator with symbol $p$, omitting in the notation the dependence on the phase $x\cdot\xi.$
\end{example}	
\begin{example}
The operators $\langle D_x\rangle$ and $|D_x|$ which are defined for all $f\in\caS(\Rd)$ by $\langle D_x\rangle^2f := 1+\sum_{j=1}^d \partial_{x_j}^2 f$ and $|D_x|^2f := \sum_{j=1}^d \partial_{x_j}^2 f$
are the pseudo-differential operators with the symbols $p(x,\xi) = \langle\xi\rangle$ and $p(x,\xi) = |\xi|$ respectively. Both are of order $1$.
\end{example}

Throughout all the article we are going to write, for the sake of brevity, FIO instead of Fourier integral operator, and PDO instead of pseudo-differential operator.

\begin{definition}\label{def:adjoint}
Given a FIO $P_\phi=p_\phi(x,D_x)$, we can define the adjoint $P_\phi^\ast$ of $P$ by $\langle P_\phi u,v\rangle:=\langle u,P_\phi^*v\rangle$ for all $u,v\in\caS(\Rd)$, and one can show that the adjoint $P_\phi^\ast$ has phase function $-\phi$ and symbol with the following asymptotic expansion:
\begin{equation}\label{eq:defp^*}
	p^*(x,\xi) = \sum_{\alpha\in\N_0^d} \frac{(-1)^{|\alpha|}}{\alpha!} \overline{\partial_{\xi}^{\alpha}\partial_{x}^{\alpha}p(x,\xi)}.
\end{equation}
\end{definition}

This allows us to generalize FIOs to a larger domain.

\begin{definition}\label{def:pseuDOonSchwartz}
Let $P_\phi=p_\phi(x,D_x)$ be a FIO on $\caS(\Rd)$. We can extend $P_\phi$ to the tempered distributions $T\in\caS'(\Rd)$ by defining for all $v\in\caS(\Rd)$
\[ \langle P_\phi T,v\rangle:=\langle T,P_\phi^\ast v\rangle.\]
\end{definition}

Recall that the \emph{Schwartz kernel} of a linear operator $A:\caD(\Rd)\to\caD'(\Rd)$ is the distribution $K_{A}\in\caD'(\Rd\times\Rd)$ given for all $u,v\in\caD(\Rd)$ by
\[ \langle K_{A},u\otimes v\rangle  = \langle Av,u\rangle.\] 

We see directly from the definition of FIO and Fubini's theorem w.r.t.\ $dx$ and $dy\dbar\xi$ above that for all $u,v\in\caS(\Rd)$
\begin{align*}
  \langle P_\phi v,u\rangle
  & = \int_\Rd (P_\phi v)(x) u(x) dx \\
  & = \int_\Rd \bigg(\int_{\Rd\times\Rd} \e^{\ii\phi(x,\xi)-\ii y\cdot\xi}p(x,\xi)v(y)dy\dbar\xi\bigg)u(x)dx \\
	& = \int_{\Rd\times\Rd}\int_\Rd \e^{\ii\phi(x,\xi)-\ii y\cdot\xi}p(x,\xi)u(x)v(y)dxdy\dbar\xi \\
	& = \left\langle O_S-\int_\Rd \e^{\ii\phi(x,\xi)-\ii y\cdot\xi}p(x,\xi)\dbar\xi,u\otimes v \right\rangle.
\end{align*}
This implies that the Schwartz kernel of a FIO $P_\phi$ is given by
\[ K_{P_\phi}(x,y) = O_S-\int_\Rd \e^{\ii\phi(x,\xi)-\ii y\cdot\xi}p(x,\xi)\dbar\xi. \]
Furthermore we see that for every $x\in\Rd$, $K_{P_\phi}(x,\cdot)\in\caS'(\Rd)$. This observation allows us to compute the Fourier transform in the second argument of the Schwartz kernel of a FIO for every $x\in\Rd$ fixed.

\begin{proposition}\label{prop:ftofschwartzkernel}
  Let $P_\phi$ be a FIO with symbol $p$ and let $K_{P_\phi}=(K_{P_\phi}(x,\cdot);x\in\Rd)$ denote its Schwartz kernel. Then the Fourier transform in the second argument of its Schwartz kernel,  $\caF_{y\mapsto\eta}K_{P_\phi}(x,\cdot)$, is given by
  \begin{equation}\label{eq:ftofschwartzkernel}
		(\caF_{y\mapsto\eta}K_{P_\phi}(x,\cdot))(\eta) = \e^{\ii\phi(x,-\eta)}p(x,-\eta).
  \end{equation}
\begin{proof}
Let $u,v\in\caS(\Rd)$. First we note that due to \eqref{eq:definitionfouriertransformSD}, the Fourier transform of $K_{P_\phi}(x,\cdot)$ is defined for all fixed $x\in\Rd$ by
\[  \big(\caF_{y\mapsto\eta}K_{P_\phi}(x,\cdot)\big)v = K_{P_\phi}(x,\cdot)\big(\caF_{\eta\mapsto y}v\big). \]
We compute using the second representation of the Schwartz kernel in \eqref{eq:definitionFIO}
\begin{align*}
	\langle \caF_{y\mapsto \eta}K_{P_\phi}v,u\rangle
	& = \int_\Rd \left(\big(\caF_{y\mapsto \eta}K_{P_\phi}(x,\cdot)\big)v\right) u(x)dx \\
	& = \int_\Rd \left(K_{P_\phi}(x,\cdot)\big(\caF_{\eta\mapsto y}v\big)\right)u(x) dx \\
	& = \int_\Rd\int_\Rd\int_\Rd \e^{\ii\phi(x,\xi)-\ii\xi\cdot y}p(x,\xi)\big(\caF_{\eta\mapsto y}v\big)(y)dy\dbar\xi u(x)dx \\
	& = \int_\Rd\int_\Rd \e^{\ii\phi(x,\xi)}p(x,\xi)\int_\Rd \e^{\ii y\cdot(-\xi)}\big(\caF_{\eta\mapsto y}v\big)(y) \dbar y d\xi u(x) dx \\
	& = \int_\Rd\int_\Rd \e^{\ii\phi(x,\xi)}p(x,\xi) v(-\xi) d\xi u(x)dx \\
	& = \int_\Rd\int_\Rd \e^{\ii\phi(x,-\eta)}p(x,-\eta) u(x)v(\eta) d\eta dx,
\end{align*}
where we have used throughout the calculation that $v$ and its Fourier transform are Schwartz functions, $K_{P_\phi}$ has a pointwise interpretation in $x$, and in the last line we have used the change of variable $\xi\mapsto-\eta$.
\end{proof}
\end{proposition}

Applying this proposition, we can show the following:

\begin{lemma}\label{lem:FTfundamentalsolution}
Let $P_\phi$ be a FIO with symbol $p$, and let $K_{P_\phi}$ denote its Schwartz kernel. Then, for every $x\in\Rd$, $K_{P_\phi}(x,\cdot)\in \caS'_r(\Rd)$.
\begin{proof}
Fix $x\in\Rd$ and $\psi\in\caD(\Rd)$. We know by \cite[p.\ 244/245]{schwartz} that the regularization of $K_P$ with a $\caC_0^\infty$-function $K_{P_\phi}(x,\cdot)\star\psi$ is an infinitely differentiable function of slow growth. We show now that $K_{P_\phi}(x,\cdot)\star\psi$ is even a Schwartz function; this implies the assertion, again by \cite[p.\ 244/245]{schwartz}. For this we take the Fourier transform of $K_{P_\phi}(x,\cdot)\star\psi$ (in the sense of distributions) and using \cite[Theorem 1.5.3(2)]{kumano-go} and Proposition \ref{prop:ftofschwartzkernel}, we conclude that
\[ \caF_{y\mapsto\eta}(K_{P_\phi}(x,\cdot)\star\psi)(\eta) = \caF_{y\mapsto\eta}K_{P_\phi}(x,\eta)\caF\psi(\eta) = \e^{\ii\phi(x,-\eta)}p(x,-\eta)\caF\psi(\eta). \]
The function of the right-hand side of the previous equality is obviously in $\caC^\infty(\Rd)$ with respect to $\eta$. The fact that $\phi$ is of order $1$ in $\eta$, $p$ is of finite order in $\eta$ and $\caF\psi$ is a Schwartz function imply that the function $\eta\mapsto\caF_{y\mapsto\eta}(T(x)\star\psi)(\eta)$ is a Schwartz function, and hence its inverse Fourier transform too. This finishes the proof.
\end{proof}
\end{lemma}

For the construction of the fundamental solution, we need to know how to multiply PDOs with FIOs. 

\begin{proposition}[\cite{kumano-go}, Theorems 10.2.1 and 10.2.2]\label{prop:productsPDOFIO}
	Let $P_\phi$ be a FIO with symbol $p\in S^{m_1}$ and let $Q$ be a PDO with symbol $q\in S^{m_2}$. Then $P_\phi Q$ and $QP_\phi$ are FIOs with phase function $\phi$ and symbols $r_1$ and $r_2$ (of order $m_1+m_2$) respectively, where $r_1$ and $r_2$ have asymptotic expansions
	\begin{align*}
		r_1(x,\xi) & \sim \sum_{\alpha\in\N_0^d} \frac{1}{\alpha!} \partial^\alpha_\xi\left(p(x,\xi)D^\alpha_x q(\tilde{\nabla}_\xi\phi(x;\xi,\xi'),\xi')\right)\bigg|_{\xi'=\xi}
	  \intertext{and}
	  r_2(x,\xi) & \sim \sum_{\alpha\in\N_0^d} \frac{1}{\alpha!} D^\alpha_{x'}\left(\partial^\alpha_{\xi'} p(x,\tilde{\nabla}_x\phi(x,x';\xi))q(x',\xi)\right)\bigg|_{x'=x},
	\end{align*}
	where \begin{align*}
  \tilde{\nabla}_x\phi(x,x';\xi) 		& = \int_0^1 \nabla_x\phi(x' + \theta(x-x'),\xi)d\theta, \\
	\tilde{\nabla}_\xi\phi(x;\xi',\xi) & = \int_0^1 \nabla_\xi\phi(x,\xi + \theta(\xi'-\xi))d\theta,
\end{align*}
respectively denote the gradient's mean value in the convex hull of $x,x'$ and $\xi,\xi'$.
\end{proposition}

The asymptotic expansions of the symbols to the second order are given by
\begin{align*}
  r_1(x,\xi) = & p(x,\xi)q(\nabla_\xi\phi(x,\xi),\xi) + \sum_{j=1}^d \partial_{\xi_j}p(x,\xi)D_{x_j}q(\nabla_\xi\phi(x,\xi),\xi) \\
  & + \frac{\ii}{2}p(x,\xi)\sum_{j,k=1}^d D_{x_j}D_{x_k}q(\nabla_\xi\phi(x,\xi),\xi)\frac{\partial^2\phi}{\partial\xi_j\partial\xi_k}(x,\xi) + r^*_1(x,\xi),
	\intertext{and}
  r_2(x,\xi) = & p(x,\nabla_x\phi(x,\xi))q(x,\xi) + \sum_{j=1}^d \partial_{\xi_j}p(x,\nabla_x\phi(x,\xi))D_{x_j}q(x,\xi) \\
  & - \frac{\ii}{2}\bigg(\sum_{j,k=1}^d \partial_{\xi_j}\partial_{\xi_k}p(x,\nabla_x\phi(x,\xi))\frac{\partial^2\phi}{\partial x_j\partial x_k}(x,\xi)\bigg)q(x,\xi) + r^*_2(x,\xi),
\end{align*}
where $r_1^*,r^*_2\in S^{m_1+m_2-2}$.

The following corollary immediately follows from Proposition \ref{prop:productsPDOFIO}.

\begin{corollary}\label{svilupprodpdo}
Let $P$ and $Q$ be PDOs with symbols $p(x,\xi)\in S^{m_1}$ and $q(x,\xi)\in S^{m_2}$. Then $PQ$ is a PDO with symbol $r(x,\xi)\in S^{m_1+m_2}$ having the asymptotic expansion
\[r(x,\xi) \sim \sum_{\alpha\in\N_0^d} \frac{1}{\alpha!} \partial^\alpha_\xi p(x,\xi)D^\alpha_x q(x,\xi).\]
\end{corollary}

A consequence of Corollary \ref{svilupprodpdo} is that the commutator $[P,Q]:=PQ-QP$ of two PDOs $P,Q$ with symbols $p\in S^{m_1}$ and $q\in S^{m_2}$ respectively is of order $m_1+m_2-1$, since the leading term of the asymptotic expansion of the symbols of both products $PQ$ and $QP$ is $p(x,\xi)q(x,\xi)$.
\\

The following Proposition states the boundedness of FIOs acting on Sobolev spaces.

\begin{proposition}[\cite{kumano-go}, Theorem 10.2.3]\label{prop:continuityFIO}
Let $P_\phi=p_\phi(x,D_x)$ and $r\in\R$. The operator $P_\phi$ defines a continuous map $H^{r+m}\longrightarrow H^r$, and there exists a constant $C=C_{r,m}>0$ and an integer $\ell\geq 0$ such that for every $u\in H^{r+m}$
\[ \|P_\phi u\|_r \leq C |p|^{(m)}_\ell \|u\|_{r+m}. \]
\end{proposition}

Finally, we give a proposition concerning the composition of $n$ FIOs, which is a simplified version of Theorem 10.6.8 in \cite{kumano-go}, referring to \cite{kumano-go} for the details.
\begin{proposition}\label{prop:productFIO}
For $1\leq j\leq n$, let $P_{j,\phi_j}$ be  FIOs with phase functions $\phi_j$ and symbols $p_j\in S^{m_j}$.
There exist a symbol $p$ of order $m=m_1+\ldots+m_n$ and a phase function $\phi$ such that
$(P_{1,\phi_1}\cdots P_{n,\phi_n})(x,D_x)=p_{\phi}(x,D_x),$
and moreover for every integer $\ell\geq 0$ there exists a constant $C_\ell>0$ and an integer $\ell'\geq 0$ such that
\begin{equation}\label{simbprodFIO}
 |p|^{(m)}_\ell\leq C_\ell^{n-1}\prod_{j=1}^n  |p_j|^{(m_j)}_{\ell'}.
 \end{equation}
\end{proposition}
The phase function of the composition $P_{1,\phi_1}\cdots P_{n,\phi_n}$ can be explicitly computed, see Section 4.5 in \cite{kumano-go}, especially formulas (5.4) and (5.5). In the statement of Proposition \ref{prop:productFIO} we focus only on formula \eqref{simbprodFIO} which is crucial in the construction of the fundamental solution, without being precise about the phase function $\phi$, which will not be used in our computations.

We conclude this section with a remark that comments on the two concepts of fundamental solutions to PDEs that we deal with in this article.

\begin{remark}\label{keyrem}
In this paper we are going to construct the fundamental solution to an initial value problem (in the sense of \cite[Section 10.7]{kumano-go}) of the form
\begin{equation}\label{initval}
	\left\{
	\begin{array}{ll}
		L(t,x,D_x)U(t,x) 	& = G(t,x),\quad (t,x)\in[0,T]\times\R^d \\
		U(0,x) 						& = U_0(x),
	\end{array}
	\right.
\end{equation}
where $L=\partial_t-\ii\caD(t,x,D_x)+\caR(t,x,D_x)$ is a square matrix of PDOs with symbols of first order, $\caD$ is the diagonal principal part and $\caR$ is some PDO of order less than $1$, that satisfies some conditions. That is, we are going to construct a family of FIOs $E(t,s)$, indexed by two time parameters $(t,s)\in\Delta_{\bar{T}}$ (see the lines above Assumption \ref{one!} for the definition), where $0<\bar{T}\leq T$ is the (modified) time horizon of the PDE, such that
 \[
	\left\{
	\begin{array}{ll}
		LE(t,s) & = 0, 	\quad(t,s)\in\Delta_{\bar{T}}\\
		E(s,s) 	&	= \id \quad s\in[0,T].
	\end{array}
	\right.
\]
Then, we can compute the solution of Problem \eqref{initval} using Duhamel's formula
\begin{equation}\label{duham}
	U(t,x) = (E(t,0)U_0)(x) + \int_0^t (E(t,s)G(s))(x)ds.
\end{equation}
This equality rewritten, using the the definition of Schwartz kernels, with $\Lambda$ denoting the Schwartz kernel of $E$, is given by
\begin{align*}
	U(t,x) & = \langle\Lambda(t,0,x,\cdot),U_0\rangle + \int_0^t\langle\Lambda(t,s,x,\cdot),G(s,\cdot)\rangle ds,
	\intertext{or, using the abuse of notation from Section \ref{sec:stochastics}}
	U(t,x) & = \int_\Rd \Lambda(t,0,x,y)U_0(y)dy + \int_0^t\int_\Rd \Lambda(t,s,x,y)G(s,y)dyds.
\end{align*}
In the case of constant coefficients, we get
\begin{align*}
	U(t,x) & = \int_\Rd \Lambda(t,x-y)U_0(y)dy + \int_0^t\int_\Rd \Lambda(t-s,x-y)G(s,y)dyds,
\end{align*}
and
$\Lambda$ can be shown to be the solution to the abstract Cauchy problem
 \[
	\begin{cases}
		L\Lambda(t,x)	& =	\delta_{0,0}, \quad (t,x)\in[0,T]\times\Rd,	\\
		\Lambda(0,x)	& =	0,  					\quad x\in\R^d,
	\end{cases}
\]
where $\delta_{0,0}$ is the space-time Dirac distribution in $(0,0)$. This concept of fundamental solution, which is fairly common in PDE theory, is the one that we refered to in the Introduction and in Section \ref{sec:stochastics}.

From now on we will refer to both concepts as fundamental solution when there is no risk of confusion. We will however make the distinction when applying Proposition \ref{prop:ftofschwartzkernel}.
\end{remark}


\subsection{Solution theory for strictly hyperbolic PDE with variable coefficients}\label{sec:appendixhyperbolic}
In this section we present in detail how to arrive at the representation for the solution to an hyperbolic PDE with time and space depending coefficients. More specifically, we focus on the Cauchy problem
\begin{equation}\label{eq:equation1}
	\begin{cases}
		P(t,x,D_t,D_x)u(t,x)	= f(t,x), &	(t,x)\in(0,T]\times\Rd,	\\
		D_t^ju(0,x)= u_j(x),  & 0\leq j\leq n-1,\ x\in\Rd,
	\end{cases}
\end{equation}
where $P$ is the partial differential operator given for $n\in\N$, $n\geq 2$, by
\begin{equation}\label{eq:higherorderPDO}
P(t,x,D_t,D_x) =D_t^n + \sum_{j=0}^{n-1}\sum_{|\alpha|\leq n-j} a_{\alpha,j}(t,x)D_x^\alpha D_t^j,
\end{equation}
and for the right-hand side we choose an arbitrary $f\in C([0,T], H^r(\Rd))$, $r\in\R$.
We assume that $P$ is strictly hyperbolic, that is the symbol of the principal part, given by
\[
p_n(t,x,\tau,\xi)=\tau^n + \sum_{j=0}^{n-1}\sum_{|\alpha|= n-j} a_{\alpha,j}(t,x)\xi^\alpha\tau^j,
\]
factorizes w.r.t. $\tau$ as
\begin{equation}\label{stricthyphigh1}
p_n(t,x,\tau,\xi)=\prod_{j=1}^n (\tau+\lambda_j(t,x,\xi)),
\end{equation}
where the $n$ characteristic roots $-\lambda_j$ of  $p_{n}$ are  such that $\lambda_j(t,x,\xi)\in\R$ for all $1\leq j\leq n$, and
\begin{equation}\label{stricthyphigh2} |\lambda_j(t,x,\xi)-\lambda_k(t,x,\xi)|\geq c|\xi|,
\end{equation}
for some $c>0$ and for all $j\neq k$. 
\\
Note that in the case of second order equations, the principal symbol becomes (omitting the dependence on $(t,x)$ of the coefficients)
\[\tau^2 + \sum_{|\alpha|= 2} a_{\alpha,0}\xi^\alpha+\sum_{|\alpha|= 1} a_{\alpha,1}\xi^\alpha\tau:= \tau^2 - \sum_{1\leq j,k\leq d} a_{j,k}\xi_j\xi_k- \sum_{1\leq j\leq d} a_{j}\xi_j\tau,\]
and we can explicitly compute the roots $-\lambda_j$; the strict hyperbolicity condition can be expressed in this case explicitly on the coefficients as
\begin{equation}\label{eq:stricthyperbolicity2}
	   \left(\sum_{j=1}^d a_{j}\xi_j\right)^2+4\sum_{j,k=1}^d a_{j,k}\xi_j\xi_k \geq C|\xi|^2,
	\end{equation}
	since \eqref{eq:stricthyperbolicity2} is enough to ensure the roots to be real and distinct for $\xi\neq 0$.
	And if $P$ is given by \eqref{intro1i}, i.e. $a_{\alpha,1}(t,x)\equiv0$ for all $|\alpha|=1$, then condition \eqref{eq:stricthyperbolicity2} is exactly \eqref{intro2}. In the general case, one cannot compute $\lambda_j$ explicitly because of the lack of a general resolution formula for higher order polynomial equations. 

In this section we want to construct a representation of the solution of \eqref{eq:equation1} by producing its fundamental solution.	
The construction presented here follows the procedure in \cite{AC,AC1,AC2} and \cite[Section 10.7]{kumano-go} and goes in three steps: first we reduce the higher-order hyperbolic equation to a first-order system, then we compute the fundamental solution to the resulting first-order system and finally, we obtain a representation formula for the fundamental solution to the higher-order equation.

\smallskip
\noindent {\it First step: Reduction to a first-order system.} Before starting with the construction, we need to point out that the factorization \eqref{stricthyphigh1} of the principal symbol of $P$ can be brought to the level of operators producing the following factorization for the principal part of the operator $P$:
\begin{equation}\label{factop}
P(t,x,D_t,D_x)=\prod_{j=1}^n (D_t+\lambda_j(t,x,D_x))+\sum_{j=0}^{n-1}S_j(t,x,D_x)D_t^j,
\end{equation}
where $S_j$ are PDOs with symbols $S_j(t,x,\xi)\in C([0,T];S^{n-j-1})$. This is only a computation, which makes use of Corollary \ref{svilupprodpdo} and works thanks to \eqref{stricthyphigh1}; for a detailed proof of \eqref{factop} we refer to \cite[Proposition 3.2, $p=1$, $n=0$]{AB}.

Let us now define the vector $V:=(v_1,\ldots,v_n)$ as follows:
\begin{equation}\label{reductiongeneral}
\begin{cases}
 	v_1 :=	\langle D_x\rangle^{n-1}u, \\
 	v_j :=	\langle D_x\rangle^{n-j}(D_t+\lambda_{j-1})\ldots(D_t + \lambda_1)u,\quad j=2,\ldots,n.
 	\end{cases}
\end{equation}
With these definitions we compute, at operator's level, for all $j=1,\ldots,n-1$
\begin{align*}
  (D_t+\lambda_j)v_j
  & = \langle D_x\rangle^{n-j}(D_t + \lambda_j)(D_t+\lambda_{j-1})\ldots(D_t+\lambda_{1})u \\
  & \phantom{= } + [\lambda_j,\langle D_x\rangle^{n-j}](D_t+\lambda_{j-1})\ldots(D_t+\lambda_{1})u \\
  & = \langle D_x\rangle v_{j+1} + [\lambda_j,\langle D_x\rangle^{n-j}]\langle D_x\rangle^{-(n-j)}v_j,
\end{align*}
and for $j=n$, by \eqref{factop} we get
\begin{align*}
  (D_t+\lambda_n)v_n =\prod_{j=1}^n (D_t+\lambda_j(t,x,D_x))u=f-\sum_{j=0}^{n-1}S_j(t,x,D_x)D_t^ju.
\end{align*}
By the reduction \eqref{reductiongeneral}, working by induction (for a proof, see \cite[formula (4.8), $p=1$]{AB}), we get
\[D_t^ju=\langle D_x\rangle^{-(n-j-1)}\sum_{\ell=1}^{j+1}S_\ell^{(0)}(t,x,D_x)v_\ell,\]
where $S_\ell^{(0)}$ are PDOs of order zero, $1\leq \ell\leq j+1$, and so
\[  (D_t+\lambda_n)v_n=f-\sum_{j=1}^{n}R_j(t,x,D_x)v_j,\]
for some PDOs $R_j$ of order zero. 

Summing up, the Cauchy problem \eqref{eq:equation1} is equivalent to the first-order system
\begin{align}\label{eq:KandR}
\begin{cases}
	\bfP(t,x,D_t,D_x)V(t,x) = G(t,x) &(t,x)\in (0,T]\times\Rd,\\
	V(0,x) = V_0(x), & x\in \Rd,
\end{cases}
\end{align}
where $V=(v_1,\ldots,v_n)^T$, $\bfP = D_t + K + \caR$, with 
\begin{equation}\label{eq:operatorhigherorder}
K(t,x,D_x)=
\begin{pmatrix}
\lambda_1& -\langle D_x\rangle	&	0	&	0	&	\cdots	&	0	&	0	\\
0	& \lambda_2& -\langle D_x\rangle	&	0	&	\cdots	&	0	&	0	\\	\vdots	&	\vdots	&	\ddots	&	\ddots	&	\cdots	&	\vdots	&	\vdots	\\
0	&	&	\cdots	&	&	&	\lambda_{n-1}& -\langle D_x\rangle\\	0	&	&	\cdots	&	&	&	0	& \lambda_{n} \\
 	\end{pmatrix},
\end{equation}
$\caR$ is a matrix of PDOs of order zero, $G=(0,\ldots,f)^T$ and
\begin{equation}\label{V0high}
V_0=(v_{0,j})_{1\leq j\leq n}^T:=\left(\sum_{\ell=0}^{j-1}S_\ell^{(n-\ell-1)}u_\ell\right)_{1\leq j\leq n}^T,
\end{equation}
with $S_\ell^{(n-\ell-1)}$ PDOs with symbols $S_\ell^{(n-\ell-1)}(x,\xi)\in S^{n-\ell-1}$, $0\leq \ell\leq n-1$, and $u_j$ the Cauchy data of the original equation \eqref{eq:equation1}.

Now we want to diagonalize the principal part of the operator matrix in \eqref{eq:operatorhigherorder}. To this end, we start working at the level of symbols and look for a diagonalizer of the bidiagonal matrix $K(t,x,\xi)$; it is easy to check that $M(t,x,\xi)=(m_{ij}(t,x,\xi))_{i,j=1,\ldots,n}$, with $m_{i,i}=1$, $m_{i,j}=0$ for $i>j$ and
\begin{equation}\label{emmij}
m_{i,j}(t,x,\xi) = \frac{(-1)^{j-1}\langle\xi\rangle^{j-i}}{\prod_{k=1}^{j-1} (\lambda_j(t,x,\xi)-\lambda_k(t,x,\xi))}
\end{equation}
for $i<j$ is a diagonalizer of $K(t,x,\xi)$. Note that the symbols $m_{i,j}$ are in $\caC([0,T],S^0)$, and that the matrix $M$ is invertible thanks to its special structure and to condition \eqref{stricthyphigh2}, and the inverse $M^{-1}$ is a matrix of symbols of order zero. 

Coming now to the level of operators, we define the operator matrix $M(t,x,D_x)$ with symbol $M(t,x,\xi)$; then we set $$W:=M^{-1}V,\;  \tilde{\bfP}:= M^{-1}\bfP M,$$ $W_0:=M^{-1}V_0$, $\tilde{G}:=M^{-1}G$ so that we obtain the system of first-order equations
\begin{equation}\label{tlachi}
\begin{cases}
	\tilde{\bfP}W = \tilde{G} &\text{ on } (0,T]\times\R^d, \\
	W(0)  = W_0  &\text{ on } \R^d,
\end{cases}
\end{equation}
where
\begin{equation}\label{fatPgeneral}
\tilde{\bfP} = D_t + K_1 + \tilde\caR,
\end{equation}
$K_1$ is a diagonal operator matrix with $\lambda_1,\ldots,\lambda_n$ as entries, and $\tilde\caR$ is an $n\times n$-operator matrix with elements in $\caC([0,T],S^0)$.

\smallskip

\noindent{\it Second step: Computing the (fundamental) solution to system \eqref{tlachi}.}
The system in \eqref{tlachi} is in the form of \cite[Section 10.7]{kumano-go}, thus by \cite[Theorem 10.7.2]{kumano-go} it admits a unique solution, that we construct here below. 

To this end, let $\phi_j = \phi_j(t,s,x,\xi)$, $1\leq j\leq n,$ be the solutions to the so-called \emph{eikonal equations} given by
\begin{equation}\label{eq:eikonalequation}
\begin{cases}
\partial_t \phi_j(t,s,x,\xi) + \lambda_j(t,x,\nabla_x\phi_j(t,s,x,\xi))= 0, & (t,s,x,\xi)\in\Delta_{\bar{T}}\times\Rd\times\Rd,\\
\phi_j(s,s,x,\xi)= x\cdot\xi, & s\in[0,\bar{T}].
\end{cases}
\end{equation}
where $x,\xi\in\Rd$ and $(t,s)\in\Delta_{\bar{T}}$, where $0<\bar{T}\leq T$ is sufficiently small. Indeed, \cite[Theorem 10.4.1]{kumano-go} states that for a sufficiently small $0<\bar{T}\leq T$ there exists a unique solution to the eikonal equations \eqref{eq:eikonalequation}, $1\leq j\leq n$. We define the operator matrix
\[ I_\phi(t,s) = \begin{pmatrix} I_{\phi_1}(t,s) & & 0 \\ & \ddots & \\ 0 & & I_{\phi_n}(t,s)\end{pmatrix}, \]
where $I_{\phi_j}$ are the FIOs with phase function $\phi_j$ and symbol $1$. From this definition together with Proposition \ref{prop:productsPDOFIO} we see that
\begin{align}\label{eq:r_0}
  & D_t I_{\phi_j} + \lambda_j(t,x,D_x)I_{\phi_j} \notag\\
  & = \int_\Rd \e^{\ii\phi_j(t,s,x,\xi)}\frac{\partial\phi_j}{\partial t}(t,s,x,\xi)\dbar\xi + \int_\Rd \e^{\ii\phi_j(t,s,x,\xi)} \lambda_j(t,x,\nabla_x \phi_j(t,s,x,\xi))\dbar\xi \notag\\
	& \phantom{=} + \int_\Rd \e^{\ii\phi_j(t,s,x,\xi)} b_{0,j}(t,s,x,\xi) \dbar\xi,
\end{align}
where $b_{0,j}(t,s)\in S^0$. The first two integral terms on the right-hand side of \eqref{eq:r_0} cancel by the definition of $\phi_j$.

Denoting by $B_{0,j}(t,s,x,D_x)$ the PDOs with symbols $b_{0,j}(t,s,x,\xi)$ in \eqref{eq:r_0}, we define the family $(W_1(t,s);(t,s)\in\Delta_{\bar{T}})$ of FIOs by
\begin{align}\label{eq:W_1}
  & W_1(t,s,x,D_x)  \\
  & := -\ii\left( \begin{pmatrix} B_{0,1}(t,s) & & 0 \\ & \ddots & \\ 0 & & B_{0,n}(t,s)\end{pmatrix} + \tilde\caR(t,x,D_x)\right)I_{\phi}(t,s,x,D_x). \nonumber
\end{align}
From \eqref{fatPgeneral}, \eqref{eq:r_0}, \eqref{eq:eikonalequation} and \eqref{eq:W_1} we obtain that
\begin{equation}\label{heart}
\tilde{\bfP}(t,x,D_x)I_{\phi}(t,s,x,D_x) = \ii W_1(t,s,x,D_x),
\end{equation}
that is $\ii W_1$ is the residual of system \eqref{tlachi} for $I_\phi$. We define then by induction the sequence of $n\times n$-matrices of FIOs, denoted by $(W_\kappa(t,s);(t,s)\in\Delta_{\bar{T}})_{\kappa\in\N}$, by
\begin{equation}\label{eq:Wn+1}
  W_{\kappa+1}(t,s,x,D_x) = \int_s^t W_1(t,\theta,x,D_x)W_\kappa(\theta,s,x,D_x)d\theta.
\end{equation}

We now claim that the operator norms of $W_\kappa$, seen as operators from the Sobolev space $H^r$ for any fixed $r$ into itself, can be estimated from above by
\begin{equation}\label{claim}
\|W_\kappa(t,s)\|\leq \frac{C_r^{\kappa-1}|t-s|^{\kappa-1}}{(\kappa-1)!} \leq \frac{C_r^{\kappa-1}\bar{T}^{\kappa-1}}{(\kappa-1)!},
\end{equation}
for all $(t,s)\in\Delta_{\bar{T}}$ and $\kappa\in\N$, where $C_r$ is a constant which only depends on the index of the Sobolev space, thanks to Propositions \ref{prop:continuityFIO} and \ref{prop:productFIO}. Indeed, to deal with the operator norms in \eqref{claim}, we need to explicitly write the matrices $W_\kappa$; an induction in \eqref{eq:Wn+1} easily shows that
 \begin{equation}\label{eq:wn+12}
 	W_\kappa(t,s) = \int_s^t\int_s^{\theta_1}\ldots\int_s^{\theta_{\kappa-2}} W_1(t,\theta_1)\ldots W_1(\theta_{\kappa-2},\theta_{\kappa-1})d\theta_{\kappa-1}\ldots d\theta_1.
 \end{equation}
 The integrand is a product of $\kappa-1$ $n\times n$-matrices  of FIOs, therefore it is an operator matrix whose entries consist of $n^{\kappa-2}$ summands of products of $\kappa-1$ FIOs. Denoting by $Q_1\ldots Q_{\kappa-1}$ one of these products, where each of the $Q_j$ is one of the $n^2$ entries of the $n\times n$-matrix of FIOs $W_1$, we have from Proposition \ref{prop:productFIO} that $Q_1\ldots Q_{\kappa-1}$ is again a FIO with symbol $\sigma_{\kappa-1}$ of order zero, and for all $\ell\in\N$ there exists $C_\ell>0$ and $\ell'\in\N_0$ such that
 \[ |\sigma_{\kappa-1}(t,\theta_1,\ldots,\theta_{\kappa-1})|_{\ell}^{(0)} \leq C_\ell^{\kappa-2} |q_1(t,\theta_1)|_{\ell'}^{(0)} \ldots |q_{\kappa-1}(\theta_{\kappa-2},\theta_{\kappa-1})|_{\ell'}^{(0)}, \]
 where for $j=1,\ldots,\kappa-1$, $q_j(t,s)$ denotes the symbol of the FIO $Q_j(t,s)$, $(t,s)\in\Delta_{\bar{T}}$.
 Now we set
 \[ \bar{\sigma} := \sup_{j=1,\ldots,\kappa-1}\sup_{(t,s)\in\Delta_{\bar{T}}} |q_j(t,s)|_{\ell'}^{(0)} < \infty, \]
 so that
 \[ |\sigma_{\kappa-1}(t,\theta_1,\ldots,\theta_{\kappa-1})|_{\ell}^{(0)} \leq C_\ell^{\kappa-2} \bar{\sigma}^{\kappa-1}.\]
 By Proposition \ref{prop:continuityFIO} applied to products of the type $Q_1\ldots Q_{\kappa-1}$ and from the previous inequality, for every $r\geq0$ there exist constants $C_r>0$ (depending only on the index of the Sobolev space) and $\ell_r\in\N_0$ such that for all $u\in H^r$
 \begin{align}\label{eq:calderonvaillancourt}
 	\|Q_1(t,\theta_1)\ldots Q_{\kappa-1}(\theta_{\kappa-2}\theta_{\kappa-1})u\|_r
 	&	\leq C_r|\sigma_{\kappa-1}(t,\theta_1,\ldots,\theta_{\kappa-1})|^{(0)}_{\ell_r}\|u\|_r \notag\\
 	&	\leq C_r C_{\ell_r}^{\kappa-2} \bar{\sigma}^{\kappa-1}\|u\|_r.
 \end{align}
 Therefore, in the operator matrix $W_1(t,\theta_1)\ldots W_1(\theta_{\kappa-2},\theta_{\kappa-1})$, the operator norm of each entry can be bounded from above by $n^{\kappa-2}C_rC_{\ell_r}^{\kappa-2}\bar{\sigma}^{\kappa-1}$, since there are $n^{\kappa-2}$ products of $\kappa-1$ FIOs. Now by \eqref{eq:wn+12} and \eqref{eq:calderonvaillancourt} we deduce that
 \begin{align}\label{eq:normWn}
   \|W_\kappa(t,s)\|
   & \leq \int_s^t\int_s^{\theta_1}\ldots\int_s^{\theta_{\kappa-2}} \|W_1(t,\theta_1)\ldots W_1(\theta_{\kappa-2},\theta_{\kappa-1})\| d\theta_{\kappa-1}\ldots d\theta_1 \notag\\
 	& \leq n^{\kappa-2}C_rC_{\ell_r}^{\kappa-2}\bar{\sigma}^{\kappa-1}\int_s^t\int_s^{\theta_1}\ldots\int_s^{\theta_{\kappa-2}} d\theta_{\kappa-1}\ldots d\theta_1 \notag\\
 	& \leq \frac{n^{\kappa-2}C_rC_{\ell_r}^{\kappa-2}\bar{\sigma}^{\kappa-1}|t-s|^{\kappa-1}}{(\kappa-1)!}= \frac{\tilde C_r^{\kappa-1}|t-s|^{\kappa-1}}{(\kappa-1)!}
 \end{align}
 for a new constant $\tilde C_r$ depending only on $r$, which yields the claim \eqref{claim}.

Now, using the estimate \eqref{claim} one can show that the sequence of FIOs defined for all $(t,s)\in\Delta_{\bar{T}}$ and all $N\in\N$ by
\begin{equation}\label{eq:En}
	E_N(t,s) = I_\phi(t,s) + \int_s^t I_\phi(t,\theta)\sum_{\kappa=1}^N W_\kappa(\theta,s)d\theta
\end{equation}
is a well-defined FIO on $H^r$ for every $r$ and converges to the well-defined operator
\begin{equation}\label{eq:E}
	E(t,s) = I_\phi(t,s) + \int_s^t I_\phi(t,\theta)\sum_{\kappa=1}^\infty W_\kappa(\theta,s)d\theta,
\end{equation}
which is the fundamental solution to the system \eqref{tlachi} in the sense that it satisfies
\begin{equation}\label{tocheck}
	\begin{cases}
		\tilde{\bfP}E(t,s) = 0 	& (t,s)\in \Delta_{\bar{T}},\\
		E(s,s)=\id 							& s\in[0,\bar{T}].
	\end{cases}
\end{equation}
Moreover, strictly hyperbolic equations are a subclass of the class of hyperbolic equations with involutive roots (see \cite{morimoto},\cite{taniguchi}), a class of equations such that the sum in \eqref{eq:E} turns out to be a finite; thus, we can assert that $E(t,s)$ is a FIO, too. Notice that $(t,s)\mapsto E(t,s)\in\caC(\Delta_{\bar{T}})$, with values in the space of the FIOs with some phase function $\phi$ and a symbol of order $0$. This can be seen from \eqref{eq:E}, because $E$ is obtained by continuous operations of operators which are continuous in $t,s$. 
By Duhamel's formula, the unique solution to system  \eqref{tlachi} is given by
\begin{align*}
W(t) = & E(t,0)W_0 + \ii\int_0^t E(t,\theta)\tilde{G}(\theta)d\theta
\\ & =E(t,0)M^{-1}(0,x,D_x)V_0 + \ii\int_0^t E(t,\theta)M^{-1}(\theta,x,D_x)G(\theta)d\theta.
\end{align*}
The entries of the vector $W(t)$ are given for $1\leq h\leq n$ by
\begin{align}\label{eq:w(t)high}
w_k(t) = & \sum_{h=1}^n\sum_{j=1}^n\left[
e_{k,h}(t,0)m_{h,j}^{-1}(0)v_{0,j}+\ii\int_0^t e_{k,h}(t,\theta)m_{h,j}^{-1}(\theta)g_j(\theta)d\theta
\right],
\end{align}
where $m_{i,k}(t,x,D_x)$ stands for a PDO with symbol $m_{i,k}(t,x,\xi)$ as in \eqref{emmij}, and $e_{k,h}(t,s)$, $1\leq k,h\leq n,$ are the entries in the operator matrix $E(t,s)$.

\smallskip

{\it Third step: Computing the fundamental solution to the equation \eqref{eq:equation1}.}
From the solution to the first-order system we can then go back to the solution to the original equation \eqref{eq:equation1}. For this we reverse all the transformations from $u$ to $V$, then from $V$ to $W$ and get
\begin{align}\label{eq:goingbacktouhigh}
u(t)	& = \langle D_x\rangle^{-(n-1)}v_1(t) = \langle D_x\rangle^{-(n-1)}\sum_{k=1}^{n}m_{i,k}(t,x,D_x)w_k(t).
\end{align}
Combining this with \eqref{eq:w(t)high} and looking at \eqref{V0high} together with the definition of $G$, we obtain the following representation for the solution $u$ of \eqref{eq:equation1}:
\begin{align}\nonumber
u(t) =& \sum_{k=1}^n\sum_{h=1}^n\sum_{j=1}^n\sum_{\ell=0}^{j-1}\langle D_x\rangle^{-(n-1)}m_{i,k}(t)e_{k,h}(t,0)m_{h,j}^{-1}(0)S_\ell^{(n-\ell-1)}u_\ell\\
\nonumber&+
\ii\sum_{k=1}^n\sum_{h=1}^n\int_0^t \langle D_x\rangle^{-(n-1)}m_{i,k}(t)e_{k,h}(t,\theta)m_{hn}^{-1}(\theta)f(\theta)d\theta
\\\label{eq:representationuhigh}
&=\sum_{\ell=0}^{n-1}T_\ell(t)u_\ell+\int_0^t T_n(t,\theta)f(\theta)d\theta
\end{align}
where $T_\ell(t)=T_\ell(t,x,D_x)$ are FIOs with symbols of order $-\ell$ for all $0\leq \ell\leq n-1$, $T_n(t,s)=T_n(t,s,x,D_x)$ a FIO with symbol of order $-(n-1)$, and $f\in C([0,T], H^r(\R^d))$, $r$ arbitrary.
Formula \eqref{eq:representationuhigh} yields the representation that we will use for instance in \eqref{eq:representationu}.

\begin{remark}
To let the construction of this section work we do not need to ask the coefficients of \eqref{eq:higherorderPDO} to be of class $\caC^\infty_b$ with respect to the spatial argument, but only to assume that they are $\caC^\ell_b$-functions in the spatial argument for a sufficiently large $\ell\in\N$. Such an $\ell$ has to be large enough such that for every $t\in[0,T]$ Proposition \ref{prop:continuityFIO} can be applied to the entries of $W_\kappa$.
This $\ell$ cannot be computed explicitly in the general case, but for a (simple enough) example it is possible to provide its precise value. This is what we are going to do in the final Section \ref{sec:weakhyperbolic}.
\end{remark}

\subsection{Stochastic second-order hyperbolic equations - the case of strict hyperbolicity}\label{sec:2ndorderhyperbolic}
In this section we consider the case where the partial differential operator $L$ in \eqref{eq:SPDE} is given by \eqref{intro1i}. More specifically, this section is devoted to the proof of the following theorem.

\begin{theorem}\label{thm:2ndorderhyperbolic} Let us consider an SPDE \eqref{eq:SPDE} where the partial differential operator $L$ is given by 
\begin{equation}\label{1i}
	L=\partial_t^2 - \sum_{j,k=1}^d a_{j,k}(t,x)\partial_{x_j}\partial_{x_k} - \sum_{j=1}^d b_j(t,x)\partial_ {x_j} - c(t,x),
\end{equation}
where for the coefficients we assume $a_{j,k}\in C^1([0,T];\caC^{\infty}_b(\R^d))$ for $1\leq j,k\leq d,$ $b_{j}\in\caC([0,T];\caC^{\infty}_b(\R^d))$ for $1\leq j\leq d$ and $c\in\caC([0,T];\caC^{\infty}_b(\R^d))$. Suppose that $L$ is a strictly hyperbolic operator, i.e.\ there exists a constant $C>0$ such that
	\begin{equation}\label{eq:stricthyperbolicity}
	  \sum_{j,k=1}^d a_{j,k}(t,x)\xi_j\xi_k \geq C|\xi|^2,
	\end{equation}
for all $(x,\xi)\in\Rd\times\Rd$. Assume for the initial conditions that $u_0\in H^r(\Rd)$ and $u_1\in H^{r-1}(\Rd)$, where $2r>d$. Furthermore, assume for the spectral measure that \eqref{eq:condition2} with $\kappa=1$ holds, and that $\sigma$ and $\gamma$ are 
such that $\gamma,\sigma\in L^2([0,T]; C_b)$, $s\mapsto\caF\sigma(s)=\nu_s \in L^2([0,T],\caM_b(\Rd))$, $s\mapsto\caF\gamma(s)=\chi_s \in L^2([0,T],\caM_b(\Rd))$.
\\
Then, for some time horizon $0<\bar{T}\leq T$, the Schwartz kernel of the FIO $T_2$ in \eqref{eq:representationu} here below satisfies Assumptions \ref{one!}, \ref{three!} and \ref{five!}, and therefore there exists a random-field solution to the SPDE \eqref{eq:SPDE} with partial differential operator given by \eqref{intro1i}.
\end{theorem}

We will start from the representation formula for the solution that we have obtained in Section \ref{sec:appendixhyperbolic}, and we will use it to show the conditions \ref{ass:A1}-\ref{ass:A5}. 

\begin{proof}[Proof of Theorem \ref{thm:2ndorderhyperbolic}] 
Let us consider the Cauchy problem 
\begin{equation}\label{eq:equation1bis}
	\begin{cases}
		L(t,x,\partial_t,\nabla_x)u(t,x)	= f(t,x), &	(t,x)\in(0,T]\times\Rd,	\\
		u(0,x)= u_0(x),  & x\in\Rd,
		\\
		\partial_tu(0,x) = u_1(x), & x\in\Rd.
	\end{cases}
\end{equation}
Using the relation $D=-\ii\partial$ we restate \eqref{eq:equation1bis} as
\begin{equation}\label{eq:equation1'}
	\begin{cases}
		P(t,x, D_t,D_x)u(t,x)	= - f(t,x), &	(t,x)\in(0,T]\times\Rd,	\\
		u(0,x)= u_0(x),  & x\in\Rd,	\\
		D_tu(0,x) = -\ii u_1(x), & x\in\Rd,
	\end{cases}
\end{equation}
with
\[ P = D_t^2 - \sum_{j,k=1}^d a_{j,k}(t,x)D_{x_j}D_{x_k} + \ii\sum_{j=1}^d b_j(t,x)D_{x_j} + c(t,x). \]
System \eqref{eq:equation1'} is a particular case of \eqref{eq:equation1}, with $n=2$, $a_{\alpha,1}\equiv 0$ for $|\alpha|=1$, $-f$ instead of $f$, $-\ii u_1$ instead of $u_1$, so the hyperbolicity condition \eqref{eq:stricthyperbolicity} corresponds exactly to \eqref{eq:stricthyperbolicity2}. Thus, following Section \ref{sec:appendixhyperbolic} we can construct by 
\eqref{eq:representationuhigh} the following representation of the solution of  \eqref{eq:equation1'} (and 
of \eqref{eq:equation1bis}):
\begin{equation}\label{eq:representationu}
	u(t) = T_0(t)u_0 + T_1(t)u_1 + \int_0^t T_2(t,s)f(s)ds,
\end{equation}
where $T_0(t)=T_0(t,x,D_x)$ is a FIO with symbol of order zero,  $T_1(t)=T_1(t,x,D_x)$ is a FIO with symbol of order $-1$, $T_2(t,s)=T_2(t,s,x,D_x)$ is a FIO with symbol of order $-1$, $u_0$ is the initial value, $u_1$ is the initial velocity, $f$ is the right-hand side.
We formally choose as the right-hand side $f(t,x):=\gamma (t,x)+\sigma(t,x)\dot F(t,x)$. Let $\Lambda(t,\cdot,x,\ast)$ denote the Schwartz kernel of the Fourier integral operator $T_2(t,s)$. So, according to \eqref{eq:mildsolutionSPDE2}, the random-field solution to the SPDE with partial differential operator as in \eqref{1i} is given by 
\begin{align}\label{eq:representationSPDE}
	u(t,x) = \big(T_0(t)u_0 + T_1(t)u_1\big)(x) & + \int_0^t\int_\Rd \Lambda(t,s,x,y)\gamma(s,y)dyds \notag \\
						& + \int_0^t\int_\Rd \Lambda(t,s,x,y)\sigma(s,y)M(ds,dy).
\end{align}
%
Now, since $T_0, T_1$ are FIOs with symbol of order $0$, $-1$ respectively, if $u_0\in H^{r}$ and $u_1\in H^{r-1}$, then $g(t):=T_0(t)u_0 + T_1(t)u_1\in H^r$. Due to the assumption that $2r>d$, we conclude by Sobolev's Embedding Theorem that $g(t)\in\caC(\Rd)$ and therefore, the pointwise evaluatation in \eqref{eq:representationSPDE} makes sense. Moreover, since $t\mapsto g(t)$ is continuous, we have that at every point $(t,x)\in[0,\bar{T}]\times\Rd$, $g(t,x)$ is well-defined, which implies \ref{ass:A5}.

Now we deal with the third term in \eqref{eq:representationu}. We see that $T_2$ (or its Schwartz kernel $\Lambda$) is the fundamental solution to the second-order SPDE with null initial conditions. By its definition in \eqref{eq:representationu}, $T_2(t,s)$ has a symbol in $\caC(\Delta_{\bar{T}},S^{-1})$ since it depends continuously on $E$ (in \eqref{eq:E}) with symbol in $\caC(\Delta_{\bar{T}},S^{0})$ and $M(t,x,D_x)$ with symbol $M(t,x,\xi)\in\caC([0,T],S^0)$ ($M$ is continuous in time because it depends continuously on the characteristic roots, and the characteristic roots of a PDE inherit the regularity with respect to time of the coefficients of the PDE). In fact, $T_2(t,s)$ is uniformly continuous with respect to $t$ and $s$ on $\Delta_{\bar T}$.

With this we can finally show the conditions \ref{ass:A1}-\ref{ass:A4}. In order to show \ref{ass:A1} and \ref{ass:A3} with $\Lambda(t,s)$ being the Schwartz kernel of $T_2(t,s)$, for each $(t,s)\in\Delta_{\bar{T}}$ we invoke Proposition \ref{prop:ftofschwartzkernel} together with \eqref{eq:pointwiseestimate} to see that
\begin{equation}\label{uuuh} 
|\caF_{y\mapsto\eta}\Lambda(t,s,x,\cdot)(\xi)|^2 = |T_2(t,s)(x,-\xi)|^2 \leq C_{t,s}\langle\xi\rangle^{-2},
\end{equation}
where $T_2(t,s)(x,-\xi)$ denotes the symbol of the FIO $T_2(t,s)$ evaluated in $(x,-\xi)$. 
Therefore the conditions \ref{ass:A1} and \ref{ass:A3} become
\begin{align*}
  \int_0^t \sup_{\zeta\in\Rd}\int_\Rd |\caF_{y\mapsto\eta}\Lambda(t,s,x,\cdot)(\eta+\zeta)|^2\mu(d\eta)|\nu_s|^2ds & \\
   \leq \int_0^t C_{t,s}|\nu_s|^2ds &\sup_{\zeta\in\Rd}\int_\Rd \frac{1}{1+|\eta+\zeta|^2}\mu(d\eta), \\
  \int_0^t \sup_{\zeta\in\Rd} |\caF_{y\mapsto\eta}\Lambda(t,s,x,\cdot)(\zeta)|^2|\chi_s|^2ds  \leq \int_0^t C_{t,s}|\chi_s|^2ds &\sup_{\zeta\in\Rd} \frac{1}{1+|\zeta|^2}.
\end{align*}
%

The constants $C_{t,s}$ can be chosen in such a way that they are continuous in $s$ and $t$, because of \eqref{uuuh} and since $T_2(t,s)$ has a symbol in $\caC(\Delta_{\bar{T}},S^{-1})$. Therefore we have that \ref{ass:A1} holds as long as \eqref{eq:condition2} holds and \ref{ass:A3} is always satisfied.

To check the two continuity conditions \ref{ass:A2} and \ref{ass:A4}, it will suffice to show that 
\begin{equation}\label{ourstar}
\sup_{r\in(s,s+h)} |\caF(\Lambda(t,s,x)-\Lambda(t,r,x))(\xi+\eta)|^2\leq \frac{C_{t,s,h}^2}{\langle\xi+\eta\rangle^2} ,
\end{equation}
with $C_{t,s,h}\to 0$ as $h\to 0$ and $C_{t,s,h}\leq C_{\bar T}$ for every $h\in [0,t-s],$ $(t,s)\in\Delta_{\bar T}$.
Indeed, if \eqref{ourstar} holds, then:
\begin{align*}
&\lim_{h\to 0}\int_0^t\bigg(\sup_{\eta\in\Rd} \int_\Rd \sup_{r\in(s,s+h)} |\caF(\Lambda(t,s,x)-\Lambda(t,r,x))(\xi+\eta)|^2 \mu(d\xi)\bigg)|\nu_s|^2ds
\\
&\leq \lim_{h\to 0}\int_0^tC_{t,s,h}^2\bigg(\sup_{\eta\in\Rd} \int_\Rd \langle\xi+\eta\rangle^{-2} \mu(d\xi)\bigg)|\nu_s|^2ds
\\
&= \bigg(\sup_{\eta\in\Rd} \int_\Rd \langle\xi+\eta\rangle^{-2} \mu(d\xi)\bigg) \lim_{h\to 0}\int_0^tC_{t,s,h}^2|\nu_s|^2ds
\\&=0
\end{align*}
via the Dominated Convergence Theorem, thanks to assumption \eqref{eq:condition2}, the fact that $|\nu_s|^2\in L^1[0,T]$ and $C_{t,s,h}\leq C_{\bar T}$. Therefore \ref{ass:A2} holds, and also \ref{ass:A4} corresponding to the particular case $\mu=\delta_0$ in \ref{ass:A2} . 
\\
So, it only remains to check that \eqref{ourstar} holds. But this follows from the uniform continuity of $s\mapsto \caF\Lambda(t,s,\cdot)(\ast)$, formula \eqref{uuuh} and \eqref{eq:pointwiseestimate}.
Indeed, the function $s\mapsto \langle\ast\rangle\caF\Lambda(t,s,\cdot)(\ast)$ is, by \eqref{uuuh}, uniformly continuous on $[0,t]$ with values in the Fr\'echet space $S^{0}(\R^{2d})$ endowed with the norm
$$||a-b||=\displaystyle\sum_{\ell=0}^\infty\frac1{2^\ell}\frac{|a-b|_\ell^{(0)}}{1+|a-b|_\ell^{(0)}}.$$
So its modulus of continuity 
\begin{align*}
\omega_{t,s}(h)=\sup_{r\in(s,s+h)} ||\langle\ast\rangle\caF\Lambda(t,s,\cdot)(\ast)-\langle\ast\rangle\caF\Lambda(t,r,\cdot)(\ast)||\to 0
\end{align*}
as $h\to 0$. By \eqref{eq:pointwiseestimate} with $m=\ell=0$ we get 
{\begin{align}\label{eq:lastgap}
\sup_{r\in (s,s+h)} & |\langle\xi+\eta\rangle\caF\Lambda(t,s,x)(\xi+\eta)-\langle\xi+\eta\rangle\caF\Lambda(t,r,x)(\xi+\eta)|
\notag\\
&\leq \sup_{r\in(s,s+h)}|\langle\ast\rangle\caF\Lambda(t,s,\cdot)(\ast)-\langle\ast\rangle\caF\Lambda(t,r,\cdot)(\ast)|_0^{(0)}\langle\xi+\eta\rangle^0
\notag\\
&= \sup_{r\in(s,s+h)} \bigg( \frac{|\langle\ast\rangle\caF\Lambda(t,s,\cdot)(\ast)-\langle\ast\rangle\caF\Lambda(t,r,\cdot)(\ast)|_0^{(0)}}{1+ |\langle\ast\rangle\caF\Lambda(t,s,\cdot)(\ast)-\langle\ast\rangle\caF\Lambda(t,r,\cdot)(\ast)|_0^{(0)}}\times
\notag\\
& \phantom{mmm} \times(1+|\langle\ast\rangle\caF\Lambda(t,s,\cdot)(\ast)-\langle\ast\rangle\caF\Lambda(t,r,\cdot)(\ast)|_0^{(0)}) \bigg)
\notag\\
&\leq \omega_{t,s}(h)(1+ \sup_{r\in(s,s+h)}|\langle\ast\rangle\caF\Lambda(t,s,\cdot)(\ast)-\langle\ast\rangle\caF\Lambda(t,r,\cdot)(\ast)|_0^{(0)})
\notag\\
&\leq \omega_{t,s}(h)(1+2C_{\bar T}),
\end{align}}
where 
\[ C_{\bar T} := \max_{0\leq s\leq t\leq \bar T}C_{t,s}< \infty,  \]
by \eqref{uuuh} and the fact that $(t,s)\mapsto C_{t,s}$ is continuous on $\Delta_{\bar T}$. Therefore, the term in the last line of \eqref{eq:lastgap} goes to zero as $h\to 0.$ Choosing the constant $C_{t,s,h}=\omega_{t,s}(h)(1+2C_{\bar T})$ we get \eqref{ourstar}.

The proof is complete.
\end{proof}

\begin{remark}
Theorem \ref{thm:2ndorderhyperbolic} holds also if we substitute the assumptions on $\gamma$ with the (presumably) more general assumption $\gamma\in C([0,T], H^{r-1})$. In the latter case, the term $\int_0^t T_2(t,s)\gamma(s)ds$ in \eqref{eq:representationu} can be put into the deterministic function $g(t)\in H^r$, and it is pointwise well defined. In the statement of Theorem \ref{thm:2ndorderhyperbolic} we choose to ask $\gamma$ as in \ref{ass:A3}, that is to consider it in a si\-mi\-lar fashion as the stochastic integral term. The reason for this is that if the problem of nonstationary nonlinear SPDEs (where $\sigma$ and $\gamma$ may depend on the solution $u$) with a general distribution as fundamental solution is solved, then the extension of the results in Theorem \ref{thm:2ndorderhyperbolic} to the nonlinear case is possible.
\end{remark}

\begin{example}[The stochastic wave equation]\label{exwave}
Consider the stochastic wave equation in the whole space $\Rd$ for any spatial dimension $d\in\N$ given by
\begin{equation}\label{eq:wave1}
  \begin{cases}
		\left(\partial_t^2 - \displaystyle\sum_{j=1}^d\partial_{x_j}^2\right)u(t,x) = \gamma(t,x) + \sigma(t,x)\dot{F}(t,x), &\text{ in } (0,T]\times\Rd, \\
    u(0,x) = u_0, 					& \text{ on } \Rd, \\
    \partial_tu(0,x) = u_1, & \text{ on } \Rd.
  \end{cases}
\end{equation}
The symbol of the wave operator is $-\tau^2+|\xi|^2$, so the characteristic roots are given by $\tau=\pm|\xi|$. Note that they do not depend on $t$ and $x$, so the corresponding PDOs $\pm|D_x|$ commute with $D_t$, $D_x$ and functions of these operators. Setting as in Section \ref{sec:2ndorderhyperbolic}
\[ 
\begin{cases}
v_1=\langle D_x\rangle\Lambda
\\
v_2 = (D_t+|D_x|)\Lambda, 
\end{cases}
\]
the equivalent first-order system \eqref{eq:KandR} becomes
\[ \left(\begin{pmatrix} D_t & 0 \\ 0 & D_t \end{pmatrix} + \begin{pmatrix} |D_x| & -\langle D_x\rangle \\ 0 & -|D_x| \end{pmatrix}\right)\begin{pmatrix} v_1 \\ v_2 \end{pmatrix} = \begin{pmatrix} 0 \\ -f \end{pmatrix}, \]
with initial conditions $v_1(0)= \langle D_x\rangle u_0$ and $v_2(0)= -\ii u_1 + |D_x|u_0$. Note that the residual term $\caR$ in \eqref{eq:KandR} is not present. Now we diagonalize this system using the matrix in \eqref{emmij}, which has the form
$$\begin{pmatrix} 1 & m(\xi) \\ 0 & 1 \end{pmatrix},\quad m(\xi)=\langle\xi\rangle/(2|\xi|).$$
So the diagonalized system becomes with the notation $W=M^{-1}V$
\begin{align}\label{eq:wave2}
  \left(\begin{pmatrix} D_t & 0 \\ 0 & D_t \end{pmatrix} + \begin{pmatrix} |D_x| & 0 \\ 0 & -|D_x| \end{pmatrix}\right)\begin{pmatrix} w_1 \\ w_2 \end{pmatrix} & = \begin{pmatrix} mf \\ -f \end{pmatrix},
\end{align}
with initial conditions $w_1(0) = \langle D_x\rangle u_0 + \ii mu_1 - m|D_x|u_0$ and $w_2(0) = -\ii u_1 + |D_x|u_0$.

Define now $\phi_\pm$ to be the solutions to the eikonal equations which have the form
\[ \partial_t\phi_\pm(t,s,x,\xi) = \mp |\nabla_x\phi_\pm(t,s,x,\xi)|, \]
with the initial condition $\phi_\pm(s,s,x,\xi)=x\cdot\xi$. One can solve these PDEs explicitly to obtain the solutions $\phi_\pm(t,s,x,\xi) = x\cdot\xi\mp(t-s)|\xi|$ for all $0\leq s\leq t\leq T$ and all $x,\xi\in\Rd$. Moreover, $\bar{T}=T$. Now we set
\[ I_\phi:= \begin{pmatrix} I_{\phi_+} & 0 \\ 0 & I_{\phi_-} \end{pmatrix}, \]
where $I_{\phi_\pm}$ is the FIO having phase function $\phi_\pm$ and symbol $1$. One can compute for all $0\leq s\leq t\leq T$ that
\begin{align*}
  \left(\begin{pmatrix} D_t & 0 \\ 0 & D_t \end{pmatrix} + \begin{pmatrix} |D_x| & 0 \\ 0 & -|D_x| \end{pmatrix}\right)I_\phi(t,s) & = 0, \\
	I_\phi(s,s) & = \id,
\end{align*}
so that $W_1$ in \eqref{eq:W_1} is identical to zero, which means that in \eqref{heart} there is no residual $\caR$ and $I_\phi$ is the fundamental solution to this first-order system. Now Duhamel's formula implies that the solution to the system \eqref{eq:wave2} is
\begin{align*}
  \begin{pmatrix} w_1(t) \\ w_2(t) \end{pmatrix}
  = & \begin{pmatrix} I_{\phi_+}(t,0)\big(\langle D_x\rangle u_0 + \ii mu_1 - m|D_x|u_0\big) \\ I_{\phi_-}(t,0)\big(- \ii u_1 + |D_x|u_0\big) \end{pmatrix} \\
    & \phantom{mmmmmmmmm} + \ii\int_0^t \begin{pmatrix} I_{\phi_+}(t,\theta)(mf)(\theta) \\ -I_{\phi_-}(t,\theta)(f)(\theta)\end{pmatrix}d\theta.
\end{align*}
The solution $u$ to the wave equation \eqref{eq:wave1} with initial conditions $u_0$ and $u_1$ and a right-hand side $f$ can so be represented by using \eqref{eq:representationu} in the following way
\[  u(t) = T_1(t)u_0 + T_2(t)u_1 + \int_0^t T_3(t,s)f(s)ds, \]
where
\begin{align*}
  T_1(t) 		& = \langle D_x\rangle^{-1}\left[I_{\phi_+}(t,0)(\langle D_x\rangle - m\lambda) + mI_{\phi_-}(t,0)\lambda\right], \\
	T_2(t)		& = \ii\langle D_x\rangle^{-1}\left[I_{\phi_+}(t,0)m - mI_{\phi_-}(t,0)\right], \\
	T_3(t,s)	&	= \ii\langle D_x\rangle^{-1}\left[I_{\phi_+}(t,s)m - mI_{\phi_-}(t,s)\right].
\end{align*}
Due to the multiplication formulas in Proposition \ref{prop:productsPDOFIO} we can compute
\begin{align*}
  T_3(t,s)
  & = \ii \int_\Rd \e^{\ii x\cdot\xi - \ii (t-s)|\xi|} \frac{1}{2|\xi|}\dbar\xi - \ii \int_\Rd \e^{\ii x\cdot\xi + \ii (t-s)|\xi|} \frac{1}{2|\xi|}\dbar\xi \\
	& = \int_\Rd	\e^{\ii x\cdot\xi} \frac{1}{|\xi|}\frac{\e^{\ii (t-s)\xi}- e^{-\ii (t-s)\xi} }{2\ii}\dbar\xi \\
	& = \int_\Rd \e^{\ii x\cdot\xi} \frac{\sin((t-s)|\xi|)}{|\xi|}\dbar\xi,
\end{align*}
and $T_2(t) = T_3(t,0)$. We can see that this is a PDO with symbol $\sin(t|\xi|)/|\xi|$. On the other hand, $T_1(t)$ becomes
\begin{align*}
  T_1(t)
  & = \int_\Rd \e^{\ii x\cdot\xi - \ii t|\xi|} \frac{1}{\langle\xi\rangle} \bigg(\langle\xi\rangle - \frac{1}{2}\frac{\langle\xi\rangle}{|\xi|}|\xi|\bigg)\dbar\xi + \int_\Rd \e^{\ii x\cdot\xi + \ii t|\xi|}\frac{1}{\langle\xi\rangle}\frac{1}{2}\frac{\langle\xi\rangle}{|\xi|}|\xi|\dbar\xi \\
  & = \frac{1}{2}\int_\Rd \e^{\ii x\cdot\xi - \ii t|\xi|}\dbar\xi + \frac{1}{2}\int_\Rd \e^{\ii x\cdot\xi + \ii t|\xi|}\dbar\xi \\
  & = \int_\Rd \e^{\ii x\cdot\xi}\frac{\e^{\ii t|\xi|} + \e^{-\ii t|\xi|}}{2}\dbar\xi \\
  & = \int_\Rd \e^{\ii x\cdot\xi}\cos(t|\xi|)\dbar\xi \\
  & = \partial_t\int_\Rd \e^{\ii x\cdot\xi}\frac{\sin(t|\xi|)}{|\xi|}\dbar\xi.
\end{align*}
We can see that in this case all the three operators involve the inverse Fourier transform of $\sin(t|\xi|)/|\xi|$ or its derivative. Now letting $u_0=u_1=0$ and the right hand side of \eqref{eq:wave1} equal to $\delta_{0,0}$ we get $u(t)=T_3(t,0)\delta_0$ and for all $v\in\caS(\Rd)$ we have
\[\langle u(t), v\rangle=\langle\int_\Rd \e^{\ii x\cdot\xi} \frac{\sin(t|\xi|)}{|\xi|}\dbar\xi,v\rangle,\]
recovering the fact that the fundamental solution is the inverse Fourier transform of $\sin(t|\xi|)/|\xi|$.
\end{example}

\subsection{Stochastic higher-order hyperbolic equations}\label{sec:higherorderhyperbolic}
In this section we give a generalization of the solution theory presented in Section \ref{sec:2ndorderhyperbolic}. We treat higher order equations of the form
\begin{equation}\label{eq:higherorderSPDE}
P(t,x,D_t,D_x)u(t,x) = \gamma(t,x)+\sigma(t,x)\dot{F}(t,x),
\end{equation}
where for $n\in\N$, $n\geq 2$, $P$ is defined in \eqref{eq:higherorderPDO}, with some suitable coefficients $a_{\alpha,j}$, see Theorem \ref{thm:morderhyperbolic} below. As in Section \ref{sec:2ndorderhyperbolic}, we assume $P$ to be strictly hyperbolic asking \eqref{stricthyphigh1} and\eqref{stricthyphigh2} to be satisfied. Note that this section provides a generalization of Section \ref{sec:2ndorderhyperbolic} also in the case $n=2$, since here we allow for terms of type $a_{\alpha,1}(t,x)D_x^\alpha D_t$ with $|\alpha|=1$.

The result of this section is the following.

\begin{theorem}\label{thm:morderhyperbolic}
Let us consider an SPDE \eqref{eq:higherorderSPDE} where the partial differential operator $P$ is of the form \eqref{eq:higherorderPDO} with coefficients $a_{\alpha,j}\in C^{n-1}([0,T];\caC^{\infty}_b(\R^d))$ for $|\alpha|= n-j$, $a_{\alpha,j}\in C([0,T];\caC^{\infty}_b(\R^d))$ for $|\alpha|< n-j$, $0\leq j\leq n-1.$
Suppose that $P$ is a strictly hyperbolic operator, i.e. \eqref{stricthyphigh1} and \eqref{stricthyphigh2} hold. Assume for the initial conditions that $D_t^{j}u(0)=:u_j\in H^{r-j}(\Rd)$ $0\leq j\leq n-1$, where $2r>d$. Furthermore, assume that $\gamma$ and $\sigma$ are as in Theorem \ref{thm:2ndorderhyperbolic} and that
\begin{equation}\label{eq:higherordercondition}
  \sup_{\eta\in\Rd}\int_\Rd\frac{1}{(1+|\xi+\eta|^2)^{n-1}}\mu(d\xi) < \infty.
\end{equation}

Then, for some time horizon $0<\bar{T}\leq T$, the random-field solution of the SPDE \eqref{eq:higherorderSPDE} with partial differential operator \eqref{eq:higherorderPDO} is well-defined.
\end{theorem}

\begin{proof} 
Let us start with the representation of the solution $u(t,x)$ of \eqref{eq:equation1} obtained in \eqref{eq:representationuhigh}:
\begin{align}\nonumber
u(t) =& \sum_{\ell=0}^{n-1}T_\ell(t)u_\ell+\int_0^t T_n(t,\theta)f(\theta)d\theta.
\end{align}
with $T_\ell(t)=T_\ell(t,x,D_x)$ FIOs of order $-\ell$ for all $0\leq \ell\leq n-1$, and $T_n(t,s)=T_n(t,s,x,D_x)$  FIO of order $-(n-1)$. 

The term $g(t)=\sum_{\ell=0}^{n-1}T_\ell(t)u_\ell\in H^r$ since $u_\ell\in H^{r-\ell}$, so this term can be treated as the corresponding one in the proof of Theorem \ref{thm:2ndorderhyperbolic}. As for the term $T_n(t,s)$, it can be handled exectly as $T_2(t,s)$ of Theorem \ref{thm:2ndorderhyperbolic}, with the only difference that $T_n(t,s)$ has order $-(n-1)$; letting $\Lambda(t,s)$  be the Schwartz kernel of $T_n(t,s)$, we come to
\begin{equation}\label{yeah} |\caF_{y\mapsto\eta}\Lambda(t,s,x,\cdot)(\xi)|^2 = |\sigma(T_n(t,s))(x,-\xi)|^2 \leq C_{t,s}\langle\xi\rangle^{-2(n-1)}, 
\end{equation}
and then the well-definedness of a random field solution follow as in the proof of  Theorem \ref{thm:2ndorderhyperbolic}.
\end{proof}

The condition \eqref{eq:higherordercondition} has already been seen in \cite{dalangmuller} when dealing with higher-order beam equations.

\subsection{Stochastic second-order hyperbolic equations - the case of weak hyperbolicity}\label{sec:weakhyperbolic}
In this section we show that the assumption of strict hyperbolicity \eqref{eq:stricthyperbolicity} on equation \eqref{1i} is an important one. In fact, if this assumption does not hold, the associated PDE might not be well posed, neither in $C^\infty$ (or in usual Sobolev spaces), see \cite{CS}, nor in weighted Sobolev spaces, see \cite{M}. However, results of "well-posedness with loss of derivatives" in Sobolev spaces can be obtained under suitable assumptions in the case of weakly hyperbolic equations, i.e.\ equations having real characteristic roots which are not necessarily distinct and separate at every time, see for example \cite{AC} and the references therein.
Here \emph{well-posedness with loss of derivatives} means that in this case the fundamental solution $E(t,s)$ results to be a Fourier integral operator of order $\delta>0$, and so by Duhamel's formula \eqref{duham} one obtains a solution $U$ which is less regular then the data, i.e.\ if the right-hand side $G$ and the initial data are in $H^r$ for some $r\geq 0$, then the solution $U$ is in $H^{r-\delta}$, as proved for instance in \cite{AC1,AC2}.

This leads to the conclusion that without the assumption in \eqref{eq:stricthyperbolicity} the Fourier transform of the fundamental solution might not behave as $\langle\xi\rangle^{-1}$ as shown in the previous section, but only as $\langle\xi\rangle^{-\kappa}$ with $\kappa\in(0,1)$.

In this section we give an example of a weakly hyperbolic equation with fundamental solution that satisfies \eqref{eq:condition2} with some $\kappa\in[0,1)$. In this section we follow the ideas from \cite{AC1}, but we have to keep a much tighter control over the constants, since their size is crucial at the end.

Let us so consider the following SPDE in spatial dimension $d=1$:
\begin{equation}\label{eq:equationweak}
  \begin{cases}
    & \left(\partial_t^2 - t^k\partial_x^2 + ct^{k\rho}\partial_x\right)u(t,x) = \dot{F}(t,x),\\
		& u(0,x) = 0,\\
		& \partial_tu(0,x) = 0,
  \end{cases}
\end{equation}
where $(t,x)\in [0,1]\times\R$, $k\in\N$ with $k\geq2$, $\rho=\frac12-\frac1k$ and $c>0$ is a constant, sufficiently small, that we will set later. Then we know by \cite[Theorem 1.1]{AC1} that the associated PDE to problem \eqref{eq:equationweak} is well-posed in Sobolev spaces with loss of derivatives, and the fundamental solution $\Lambda$ exists. Therefore, the solution to problem \eqref{eq:equationweak} is defined by
\[ u(t,x) = \int_0^t\int_\R \Lambda(t,s,x-y)M(ds,dy), \]
since the coefficients and the right-hand side of the PDE corresponding to \eqref{eq:equationweak} do not depend on the spatial argument. In this section we will construct the fundamental solution $\Lambda$ and derive an upper estimate as in \eqref{eq:condition2}.

Equation \eqref{eq:equationweak} can be reformulated using $D=-\ii\partial$:
\[	\begin{cases}
			&\left(D_t^2 - t^kD_x^2 - c\ii t^{k\rho} D_x\right)u(t,x) = - f(t,x), \\
			& u(0,x) = 0,\\
			&D_tu(0,x) = 0,
		\end{cases} \]
where the right-hand side is formally given by $f=\dot F$.

We follow the ideas of Section \ref{sec:appendixhyperbolic}; since we deal with a second order equation, we explicitly compute the characteristic roots of the partial differential operator in \eqref{eq:equationweak}. Its principal symbol is given by $\tau^2-t^k\xi^2$, so the characteristic roots are $\lambda(t,\xi) = \pm t^{k/2}|\xi|$ for all $(t,x)\in[0,1]\times\R$. Since these roots coincide (and vanish) at $t=0$ for every fixed $\xi\neq 0$ (i.e. equation \eqref{eq:equationweak} is not strictly hyperbolic), to let the ideas of Section \ref{sec:appendixhyperbolic} work we separate the roots by defining the following "approximated characteristic roots":
\[ \tilde{\lambda}(t,\xi) := \pm \sqrt{t^k+\langle\xi\rangle^{-2}}\cdot|\xi| = \pm\sqrt{1+t^k\langle\xi\rangle^2}\cdot\langle\xi\rangle^{-1}|\xi|. \]
Moreover, we set
\[ \zeta(t,\xi):=\sqrt{1+t^k\langle\xi\rangle^2}, \]
and easily see that $\zeta\in\caC([0,1],S^1)$.

Now we define, using the approximated characteristic roots instead of the characteristic roots (compare with \eqref{reductiongeneral} in the case of strict hyperbolicity)
\[
	\begin{cases}
		v_1(t,x) & := \zeta(t,D_x)u(t,x) \\
		v_2(t,x) & := (D_t + \tilde{\lambda}(t,D_x))u(t,x),
	\end{cases}
\]
and by performing similar calculations as in Section \ref{sec:appendixhyperbolic}, we obtain
\begin{align*}
	\big(D_t + \tilde{\lambda}(t,D_x)\big)v_1 
	& = \zeta(t,D_x)\big(D_t + \tilde{\lambda}(t,D_x)\big)u - \ii\frac{\partial\zeta}{\partial t}(t,D_x)u \\
	& = \zeta(t,D_x)v_2 - R_0v_1,
\end{align*}
where $R_0=R_0(t,D_x)$ is the PDO with symbol
\begin{equation}\label{eq:R0}
   r_0(t,\xi) = \frac{\ii kt^{k-1}\langle\xi\rangle^2}{2(1+t^k\langle\xi\rangle^2)} = \frac{\ii kt^{k-1}}{2(t^k+\langle\xi\rangle^{-2})},
\end{equation}
and
\[\big(D_t -  \tilde{\lambda}(t,D_x)\big)v_2 = -f(t,x) - N_0(t,D_x)v_1,\]
where $N_0(t,D_x)$ is a PDO with symbol
\begin{equation}\label{eq:N0}
  n_0(t,\xi) = -\frac{c\ii t^{k\rho} \xi}{\langle \xi\rangle\sqrt{t^k + \langle\xi\rangle^{-2}}} + \frac{\ii kt^{k-1}|\xi|}{2(t^k + \langle \xi\rangle^{-2})\langle \xi\rangle} + \frac{\xi^2}{\langle \xi\rangle^2\zeta(t,\xi)}.
\end{equation}
So we obtain the equivalent first-order system
\begin{equation}\label{eq:sec4system1}
	\begin{cases}
		\bfP V(t)		& =	G,\quad\text{in } [0,1],\\
		V(0)				& = 0,
	\end{cases}
\end{equation}
where $V=(v_1,v_2)$, $G=(0,-f)$, and
\[ \bfP = D_t + \begin{pmatrix}\tilde{\lambda}(t,D_x) & -\zeta(t,D_x) \\ 0 & -\tilde{\lambda}(t,D_x) \end{pmatrix} + \begin{pmatrix} R_0(t,D_x) & 0 \\ N_0(t,D_x) & 0 \end{pmatrix}. \]

Now, before diagonalizing the system, we investigate the order of the two symbols in \eqref{eq:R0} and \eqref{eq:N0}. For this we need three lemmas, the first one being an important integral inequality.

\begin{lemma}\label{lem:integralinequality}
	For all $\alpha,\beta,\delta>0$ we have
  \[ \int_0^1 \frac{t^\alpha}{(t^\delta + \langle\xi\rangle^{-2})^\beta} dt
  \leq \begin{cases}
  				C_{\alpha,\beta,\delta} & \text{for }\alpha-\beta\delta>-1, \\
         \frac{1}{\alpha+1} + \log\big(\langle\xi\rangle^{2\beta/(\alpha+1)}\big) & \text{for }\alpha-\beta\delta=-1, \\
         C_{\alpha,\beta,\delta}\langle\xi\rangle^{2(\beta\delta-\alpha-1)/\delta} & \text{for }\alpha-\beta\delta<-1.
       \end{cases} \]
\begin{proof}
  By separating the domain of integration into $[0,h]$ and $[h,1]$, where $h:= \langle\xi\rangle^{-2\beta/(\alpha+1)}$ we obtain
  \[ \int_0^t \frac{t^\alpha}{(t^\delta + \langle\xi\rangle^{-2})^\beta} dt \leq \langle\xi\rangle^{2\beta}\int_0^h t^\alpha dt + \int_h^1 t^{\alpha-\beta\delta}dt \leq \frac{1}{\alpha+1} + \frac{1}{\alpha-\beta\delta+1}, \]
  if $\alpha-\beta\delta>-1$; if $\alpha-\beta\delta=-1$, then
  \[ \int_0^t \frac{t^\alpha}{(t^\delta + \langle\xi\rangle^{-2})^\beta} dt \leq \frac{1}{\alpha+1} + \log\big(h^{-1}\big).\]
  In the case when $\alpha-\beta\delta<-1$, we obtain by the change of variable $t^\delta + \langle\xi\rangle^{-2}\mapsto s$
\begin{align*}
  \int_0^1 \frac{t^\alpha}{(t^\delta + \langle\xi\rangle^{-2})^\beta} dt
  & = \frac{1}{\delta}\int_{\langle\xi\rangle^{-2}}^{1+\langle\xi\rangle^{-2}} \frac{(s-\langle\xi\rangle^{-2})^{\alpha/\delta+1/\delta-1}}{s^\beta} ds \\
  & \leq  \frac{1}{\delta}\int_{\langle\xi\rangle^{-2}}^{1+\langle\xi\rangle^{-2}} s^{(\alpha+1-\delta-\beta\delta)/\delta} ds \\
  & = \frac{1}{\alpha+1-\beta\delta}\Big((1+\langle\xi\rangle^{-2})^{(\alpha+1-\beta\delta)/\delta} - \langle\xi\rangle^{-2(\alpha+1-\beta\delta)/\delta}\Big) \\
  & \leq \frac{1}{\beta\delta-\alpha-1}\langle\xi\rangle^{2(\beta\delta-\alpha-1)/\delta}. \qedhere
\end{align*}
\end{proof}
\end{lemma}

The other two lemma give bounds on the derivatives of some of the terms in the symbols of $R_0$ and $N_0$.

\begin{lemma}\label{lem:4.2}
	For all $l\in\N_0$, and $j\in\{0,1\}$
	\[ \int_0^1\bigg|\partial_\xi^l \frac{t^{k-j}}{t^k+\langle\xi\rangle^{-2}}\bigg|dt \leq C_l\langle\xi\rangle^{-l}, \]
	where $C_l\leq l\cdot 2^{l+2}$.
\begin{proof}
An induction shows that the partial derivatives can be bounded form above by
\[ \partial_\xi^l \frac{t^{k-j}}{t^k+\langle\xi\rangle^{-2}} \leq \tilde{C}_l\sum_{m=1}^l \langle\xi\rangle^{-l-2m} \frac{t^{k-j}}{(t^k + \langle\xi\rangle^{-2})^{m+1}}, \]
where $\tilde{C}_l\leq 2^{l+2}$. Then integrating this over $[0,1]$ and using Lemma \ref{lem:integralinequality} with $\alpha=k-j$, $\beta=m+1$ and $\delta=k$, we get
\begin{align*}
	\int_0^1\bigg|\partial_\xi^l \frac{1}{t^k+\langle\xi\rangle^{-2}}\bigg|dt
	& \leq \tilde{C}_l\sum_{m=1}^l \langle\xi\rangle^{-l-2m} \int_0^1 \frac{t^{k-j}}{(t^k + \langle\xi\rangle^{-2})^{m+1}} dt \\
	& \leq \tilde{C}_l\langle\xi\rangle^{-l+2\frac{j-1}{k}}\sum_{m=1}^l\frac{1}{mk+j-1} \\
	& \leq l\cdot \tilde{C}_l \langle\xi\rangle^{-l}. \qedhere
\end{align*}
\end{proof}
\end{lemma}

This lemma implies that $r_0$ in \eqref{eq:R0} is of order zero. Moreover, we immediately see that the third term on the right-hand side of \eqref{eq:N0} is of order $-1$ since $\zeta$ is of order $1$. The second symbol is of order zero, being the product of the two symbols of order zero $kt^{k-1}/(t^k + \langle \xi\rangle^{-2})$ and $|\xi|/\langle \xi\rangle$.

The following lemma is needed to investigate the first term on the right-hand side of \eqref{eq:N0}.

\begin{lemma}\label{lem:4.3}
  For all $l\in\N_0$,
  \begin{equation}\label{eq:lem43}
		\int_0^1 \bigg|\partial^l_\xi \frac{t^{k\rho}\xi}{\sqrt{1+t^k\langle\xi\rangle^2}}\bigg|dt \leq C_l^\prime\langle\xi\rangle^{-l}\log(1+\langle\xi\rangle),
\end{equation}
  for some constant $C_l^\prime<2l\cdot l!\cdot (2l-1)!!$.
\begin{proof}
First we see what happens at the level of $l=0$. Then the integrand in \eqref{eq:lem43} can be written as
\begin{equation}\label{eq:strangeterm}
	\frac{t^{k\rho} \xi}{\langle\xi\rangle\sqrt{t^k + \langle \xi\rangle^{-2}}} = \frac{\xi}{\langle\xi\rangle}\bigg(\frac{t^k}{t^k+\langle\xi\rangle^{-2}}\bigg)^\rho\frac{1}{(t^k + \langle\xi\rangle^{-2})^{1/k}},
\end{equation}
where, with the help of Lemma \ref{lem:4.2}, the first two terms can be seen to be symbols of order zero. The third term is a symbol in $S^{2/k}\subseteq S^1$, and using Lemma \ref{lem:integralinequality} with $\alpha=0$, $\beta=k^{-1}$ and $\delta=k$, we get
\begin{align}\label{eq:integralstrangeterm}
  \int_0^1 \frac{1}{(t^k + \langle\xi\rangle^{-2})^{1/k}}dt
  & \leq 1 + \log(\langle\xi\rangle^{2/k})\notag\\
  & \leq 1+\log\langle\xi\rangle.
\end{align}
The next step is to investigate the derivatives of the term in \eqref{eq:strangeterm}. We define
\begin{align*}
	\theta(t,\xi)
	:= \frac{t^{k\rho}}{\sqrt{1+t^k\langle\xi\rangle^2}}
	& = \frac{t^{k\rho}}{\langle\xi\rangle\sqrt{\langle\xi\rangle^{-2}+t^k}} \\
	& = \bigg(\frac{t^k}{\langle\xi\rangle^{-2} + t^k}\bigg)^\rho\frac{1}{(\langle\xi\rangle^{-2} + t^k)^{1/k}}\frac{1}{\langle\xi\rangle},
\end{align*}
and for all $l\in\N_0$
\begin{align*}
  \Theta_l(t,\xi)
  & := (-1)^l(2l-1)!! \frac{t^{k(\rho+l)}}{(1+t^k\langle\xi\rangle^2)^{(2l+1)/2}} \\
  & = (-1)^l(2l-1)!! \frac{t^{k(\rho+l)}}{\langle\xi\rangle^{2l+1}(\langle\xi\rangle^{-2}+t^k)^{(2l+1)/2}} \\
  & = (-1)^l(2l-1)!! \bigg(\frac{t^k}{\langle\xi\rangle^{-2}+t^k}\bigg)^{\rho+l}\frac{1}{(\langle\xi\rangle^{-2} + t^k)^{1/k}}\frac{1}{\langle\xi\rangle^{2l+1}},
\end{align*}
where $n!!$ denotes the odd factorial of an odd number, i.e.\ $n!! = n(n-2)\ldots3\cdot1$, and $(-1)!!:=1$. We have that $\Theta_0(t,\xi)=\theta(t,\xi)$ and
\[ \partial_\xi\Theta_l(t,\xi) = \Theta_{l+1}(t,\xi)\langle\xi\rangle\partial_\xi\langle\xi\rangle = \Theta_{l+1}(t,\xi)\xi. \]
Set furthermore
\begin{equation}\label{brackets} \tilde{\Theta}_l(t,\xi) := (-1)^l(2l-1)!! \bigg(\frac{t^k}{\langle\xi\rangle^{-2}+t^k}\bigg)^{\rho+l}\frac{1}{(\langle\xi\rangle^{-2} + t^k)^{1/k}}, 
\end{equation}
so $$\tilde{\Theta}_l(t,\xi)=\Theta_l(t,\xi)\langle\xi\rangle^{2l+1}.$$ Note that the term in the brackets in \eqref{brackets} is bounded by $1$.

In the sequel, we will deal with symbols of the form $p_{a_1,a_2}(\xi):=\xi^{a_1}\langle\xi\rangle^{a_2}\in S^{a_1+a_2}$, where $a_1,-a_2\in\N_0$. Note that the symbols $\tilde{\Theta}_l(t,\xi)$ and $p_{a_1,a_2}$ are bounded by $(2l-1)!!$ and $1$ respectively. With this preparation, we can evaluate the derivatives of the function on the left-hand side of \eqref{eq:strangeterm}, which is equal to $\theta(t,\xi)\xi$. Its derivatives are given by the following formula. Set $l^*:=l/2+1$ if $l$ is even and $l^*:=(l+1)/2+1$ if $l$ is odd, i.e.\ $l^*=\lceil l/2+1\rceil$. Then
\begin{align*}
  \partial^l_\xi\big(\theta(t,\xi)\xi\big) = \sum_{j=l-l^*+1}^l C_{l,j}\tilde{\Theta}_l(t,\xi)p_{2j-(l-1),2j+1}(t,\xi),
\end{align*}
where $C_{l,j}\in\N_0$ are constants which can be recursively computed. We have for instance $C_{l,l}=1$, $C_{l,l-1}=\sum_{j=1}^l j$, and all $C_{l,j}\leq 2 l!$. Therefore, using \eqref{eq:integralstrangeterm} we have that  its derivatives satisfy the following integral inequality
\begin{align}\label{eq:boundr0}
  \int_0^1 \big|\partial^l_\xi\big(\theta(t,\xi)\xi\big)\big| dt
  & \leq 2l\cdot l!\cdot (2l-1)!! \cdot\langle\xi\rangle^{-l}\int_0^1 \frac{1}{(t^k+\langle\xi\rangle)^{1/k}} dt \notag\\
  & \leq 2l\cdot l!\cdot (2l-1)!! \cdot\langle\xi\rangle^{-l}(1+\log\langle\xi\rangle) \notag\\
  & \leq C_l^\prime\langle\xi\rangle^{-l}\log(1+\langle\xi\rangle),
\end{align}
where $C_l^\prime := 2l\cdot l!\cdot (2l-1)!!$.
\end{proof}
\end{lemma}

This lemma implies that the first term on the right-hand side of \eqref{eq:N0} is of order $2/k$ and satisfies the integral inequality \eqref{eq:lem43}.

The next step is to diagonalize the system \eqref{eq:sec4system1} using the matrix in \eqref{emmij}, which in this case has the form
\[
\left(\begin{array}{cc}
1&m(t,\xi)
\\
0&1
\end{array}
\right)\]
with
\[ m(t,\xi):=\frac{\zeta(t,\xi)}{2\tilde{\lambda}(t,\xi)} = \frac{\sqrt{1+t^k\langle\xi\rangle^2}}{2\sqrt{\langle\xi\rangle^{-2}+t^k}|\xi|} = \frac{\langle\xi\rangle}{2|\xi|}. \]
Then, with the change of variable $W:=M^{-1}V$ the system becomes
\begin{equation}\label{syssec4}
	\begin{cases}
 		\tilde{\bfP}W(t,x) & = \tilde G(t,x),\ (t,x)\in(0,1]\times\R, \\
		W(0,x) & = 0,	\quad x\in\R,
	\end{cases}
\end{equation}
with $\tilde G=(mf,-f)^T$ and
\[ \tilde\bfP(t,x,D_x) = \begin{pmatrix} D_t & 0 \\ 0 & D_t\end{pmatrix} + \begin{pmatrix} \tilde{\lambda}(t,x,D_x) & 0 \\ 0 & -\tilde{\lambda}(t,x,D_x)\end{pmatrix} +\tilde\caR(t,x,D_x), \]
with
\begin{equation}\label{rsec4} \tilde\caR = \begin{pmatrix} R_0 - mN_0 & R_0m-mN_0m \\ N_0 & N_0m \end{pmatrix}. 
\end{equation}
Note that here $\tilde\caR$ is a matrix of PDOs of order $2/k$, not of order zero as in the strictly hyperbolic case, see \eqref{fatPgeneral}. 

It is important to remark here that in Section \ref{sec:appendixhyperbolic} we constructed the FIO $W_1$ in \eqref{eq:W_1} with the same order as $\tilde\caR$ and the behavior of $W_1$ was used to obtain the well-definedness of the symbol of $E_N$ in \eqref{eq:En} and its order.
More precisely,  it was crucial to have a uniform in time estimate of $\int_s^t W_1(t,\theta)d\theta$.

Here we want to follow the same ideas, so now we derive an integral estimate for the symbols of the four operators in the matrix \eqref{rsec4}. 
It can be easily checked that the symbols of the operators $mN_0$, $N_0m$ and $mN_0m$ satisfy the same integral estimate as in \eqref{eq:boundr0} with the same constants $C_l^\prime$, which was derived in Lemma \ref{lem:4.3}. For the latter estimate we consider, as in Definition \ref{CPsymbol}, the symbol only for $\xi$ outside the ball with radius $R>1$. The symbols of the other two operators $R_0$ and $R_0m$ satisfy the same integral inequality as in Lemma \ref{lem:4.2}, with the constant $C_l\cdot k$. The symbols of the four PDOs in $\tilde\caR$, denoted by $\tilde r_{i,j}$ for $i,j\in\{1,2\}$, satisfy therefore
\begin{align*}
\int_0^1 |\partial_\xi^l \tilde r_{i,j}(t,\xi)|dt &\leq \big(cC_l^\prime\log(1+\langle\xi\rangle)+C_l\cdot k\big)\langle\xi\rangle^{-l} 
\\
&\leq C_{k,l}^\prime\big(1+c\log(1+\langle\xi\rangle)\langle\xi\rangle^{-l},
\end{align*}
where $C_{k,l}^\prime\leq C_l^\prime + C_l\cdot k$. Then for $\xi$ outside a sufficiently large ball, whose radius may depend on $c$, $c\log(1+\langle\xi\rangle)$ dominates the constant $1$. 

Now, to construct the fundamental solution $E$ to the system \eqref{syssec4} we have to substitute the approximate characteristic roots by the true ones, rewriting the operator $\tilde\bfP$ in the form
\[ \tilde\bfP(t,x,D_x) = \begin{pmatrix} D_t & 0 \\ 0 & D_t\end{pmatrix} + \begin{pmatrix} \lambda(t,D_x) & 0 \\ 0 & -\lambda(t,D_x)\end{pmatrix} + \tilde{\tilde{{\caR}}}(t,x,D_x), \]
with 
\[ \tilde{\tilde{\caR}} =\tilde\caR+( \tilde{\lambda}-\lambda ) \begin{pmatrix}1 & 0 \\ 0 & -1 \end{pmatrix}. \]
We notice that
\[(\tilde\lambda-\lambda)(t,\xi)=\big(\sqrt{t^k + \langle\xi\rangle^{-2}} - t^{k/2}\big)|\xi|=\displaystyle\frac{\langle\xi\rangle^{-2}|\xi|}{\sqrt{t^k + \langle\xi\rangle^{-2}} + t^{k/2}}\]
is a symbol of order zero, and by Lemma \ref{lem:integralinequality} with $\alpha=0, \beta=1/2$, $\delta=k$ we have (since $k> 2$) 
\begin{align*}
  \int_0^1\vert(\tilde\lambda-\lambda)(t,\xi)\vert dt
  & \leq \displaystyle\frac{2}{k-2}\langle\xi\rangle^{1-(2k)^{-1}}\langle\xi\rangle^{-2}|\xi|\leq  \displaystyle\frac{2}{k-2};
\end{align*}
similarly, we can compute
\[ \int_0^1 \bigg|\partial^l_\xi\big(\tilde{\lambda}(t,\xi) - \lambda(t,\xi)\big)\bigg|dt \leq \frac{2}{k-2}C_{l}\langle\xi\rangle^{-l}, \]
with $C_{l}$ computed in Lemma \ref{lem:4.2}. This means that $\tilde{\tilde{\caR}}$ is again a matrix of PDOs of order $2/k$ and the symbols of its entries satisfy
 \begin{align}\nonumber
\int_0^1 |\partial_\xi^l \tilde{\tilde{ r}}_{i,j}(t,\xi)|dt &\leq \bigg(2cC_{k,l}^\prime\log(1+\langle\xi\rangle)+C_l\frac{2}{k-2}\bigg)\langle\xi\rangle^{-l} 
\\
\label{remainderintegral}
&\leq cC_{k,l}\log(1+\langle\xi\rangle)\langle\xi\rangle^{-l}, 
\end{align}
with $C_{k,l}\leq 2C_{k,l}^\prime + 2C_l/(k-2)$.

We define now $I_\phi, \phi_{\pm}, W_1, W_n, E_N, E$ as in Section \ref{sec:appendixhyperbolic}, with the difference that here we have $\tilde{\tilde{\caR}}$ satisfying \eqref{remainderintegral} (instead of $\tilde\caR$ of order zero as in \eqref{eq:W_1}). 

We have by Proposition \ref{prop:continuityFIO} that for every given $r>0$ there exist a constant $\tilde C_r>0$ and an integer $l_0\in\N_0$ sufficiently large such that when setting
\[ \delta_0 := \sup_{l\leq l_0} C_{k,l} = C_{k,l_0} \leq 2l_0\cdot l_0!\cdot(2l_0- 1)!!, \]
we get for every $\ell\leq\ell_0$ 
\small
\begin{align}
  & |\partial^\ell_\xi\sigma(E_N(t,s))(x,\xi)| \notag\\
  & \leq \sum_{n=1}^N\int_s^t\int_s^\theta\int_s^{\theta_1}\hskip-0.4cm\ldots\int_s^{\theta_{n-2}} \big|\partial^\ell_\xi\sigma\big(W_1(t,\theta_1)\ldots W_1(\theta_{n-2},\theta_{n-1})\big)(x,\xi)\big| d\theta_{n-1}\ldots d\theta_1d\theta \notag\\
  & \leq \langle \xi\rangle^{-\ell}(2C_{\ell_r})^{-1} \sum_{n=1}^N \frac{\left(2C_{\ell_r}c\tilde{C}_{r}\delta_0(1+\log\langle\xi\rangle)\right)^{n-1}}{(n-1)!}\leq C_{\ell_r,r}\langle\xi\rangle^{\delta-\ell},
\end{align}
\normalsize
where
$$\delta:=2C_{\ell_r}c\tilde{C}_{r}\delta_0,$$
with a new constant $C_{\ell_r,r}>0$. Thus the operator $E$, defined as in \eqref{eq:E} as the limit of $E_N$ as $N\to\infty$, is continuous from $H^{s+\delta}$ to $H^s$ for every $s\in\R$, i.e. it is an operator of order $\delta>0$. It's now easy to check that $E$ is the fundamental solution to the system
\[
	\begin{cases}
		\tilde\bfP E(t,s) = 0 & (t,s)\in\triangle_{\bar T}\\
		E(s,s)=\id & s\in [0,\bar T]
	\end{cases}
\]
for some $\bar{T}<1$. For more precise computations (in a general setting) we refer to \cite[Theorem 3.1]{AC1} or \cite[Theorem 3.1]{AC2}.

If we set $0<c<(2C_{\ell_r}\tilde{C}_r\delta_0)^{-1}$, then $0<\delta<1$. We compute the order of the fundamental solution to \eqref{eq:equationweak} reversing all the transformations from $u$ to $w$, i.e.
\[ u(t) = \zeta(t,D_x)^{-1}\big(w_1(t) + mw_2(t)\big). \]
So, by substituting the above expressions for $w$ into this equation we obtain that the fundamental solution to \eqref{eq:equationweak} is given by
\begin{align*}
    T_2(t,s)	& := \ii \zeta(t)^{-1}\big(e_{1,1}(t,s)m - e_{1,2}(t,s) + me_{2,1}(t,s)m - me_{2,2}(t,s)\big),
\end{align*}
where $e_{i,j}$ are the entries of $E$. Since $\zeta$ is an operator of order $1$ and $m$ is an operator of order $0$, we have that the fundamental solution $T_3(t,s)$ is an operator of order $\delta-1\in(-1,0)$. Therefore, a sufficient condition for the assumption \ref{ass:A1} on the spectral measure for the well-definedness of a random-field solution is
\[ \sup_{\eta\in\R} \int_\Rd \frac{1}{(1+|\xi+\eta|^2)^{1-\delta}}\mu(d\xi) < \infty. \]
The other conditions for the existence can be shown with similar arguments as the ones in the proof of Theorem \ref{thm:2ndorderhyperbolic}. We have so proved the following:
\begin{theorem} Let us consider the SPDE \eqref{eq:equationweak} If the spectral measure satisfies \eqref{eq:condition2} for some $\kappa<1$, then for some time horizon $0<\bar{T}\leq T$ 
the Schwartz kernel of the FIO $T_2$ satisfies Assumptions \ref{one!} and \ref{three!}, and therefore a random-field solution to the SPDE \eqref{eq:equationweak} is well-defined.
\end{theorem}

\end{document}